\documentclass[12pt]{amsart}
\usepackage[utf8]{inputenc}
\usepackage[initials]{amsrefs}
\usepackage[
        colorlinks,
]{hyperref}

\usepackage{fullpage,appendix,amssymb,amsmath}
\newcommand{\Q}{\mathbb Q}
\newcommand{\R}{\mathbb R}
\newcommand{\F}{\mathbb F}

\newcommand{\ord}{{\rm ord}}
\newcommand{\C}{\mathbb C}
\newcommand{\Z}{\mathbb Z}
\newcommand{\Irr}{{\rm Irr}}
\newcommand{\Gal}{{\rm Gal}}
\newcommand{\e}{{\rm e}}
\newcommand{\eps}{\varepsilon}
\renewcommand{\d}{{\rm d}}
\newcommand{\rd}{{\rm rd}}

\newcommand{\p}{\mathfrak p}
\renewcommand{\k}{\mathbf k}

\newcommand{\bchi}{\boldsymbol{\chi}}

\newcommand{\bfn}{\boldsymbol{n}}
\newcommand{\bfa}{\boldsymbol{a}}

\renewcommand{\Re}{{\mathfrak R}{\rm e}}
\renewcommand{\Im}{{\mathfrak I}{\rm m}}

\newtheorem{X}{X}[section]
\newtheorem{corollary}[X]{Corollary}

\newtheorem{lemma}[X]{Lemma}
\newtheorem{proposition}[X]{Proposition}
\newtheorem{theorem}[X]{Theorem}

\theoremstyle{remark}
\newtheorem{remark}[X]{Remark}

\theoremstyle{remark}
\newtheorem*{example}{Example}

\def\sums{\mathop{\sum \Bigl.^{*}}\limits}

\title{Moments in the Chebotarev density theorem: \\general class functions}
\author{R\'egis de la Bret\`eche}
\address{Universit\'e Paris Cit\'e, Sorbonne Universit\'e, CNRS,
Institut de Math\'ematiques de Jussieu-Paris Rive Gauche,
 F-75013 Paris,
France}
\email{regis.delabreteche@imj-prg.fr}
\author{Daniel Fiorilli}
\address{Univ. Paris-Saclay, CNRS, Laboratoire de math\'ematiques d'Orsay, 91405, Orsay, France.}
\email{daniel.fiorilli@universite-paris-saclay.fr}
\author{Florent Jouve}
\address{Univ. Bordeaux, CNRS, Bordeaux INP, IMB, UMR 5251, F-33400, Talence, France.}
\email{florent.jouve@math.u-bordeaux.fr}
\date{\today}

\dedicatory{À la mémoire de Joël Bellaïche}
\begin{document}

\begin{abstract}

In this paper we find lower bounds on higher moments of the error term in the Chebotarev density theorem. Inspired by the work of Bella\"{\i}che, we consider general class functions and prove bounds which depend on norms associated to these functions. Our bounds also involve the ramification and Galois theoretical information of the underlying extension $L/K$. Under a natural condition on class functions (which appeared in earlier work), we obtain that those moments are at least Gaussian. The key tools in our approach are the application of positivity in the explicit formula followed by combinatorics on zeros of Artin $L$-functions (which generalize previous work), as well as precise bounds on Artin conductors.

\end{abstract}

\maketitle
\nocite{*}
\section{Introduction}

The study of the error term in the Chebotarev density theorem has a long history and is critical in many applications. If $L/K$ is a Galois extension of number fields, $G=\Gal(L/K)$ and $C \subset G$ is a conjugacy class, then this theorem states that as $x\rightarrow \infty$
$$ \pi_C(x;L/K) := \sum_{\substack{ \p \triangleleft \mathcal O_K \\ \mathcal N \p \leq x \\ \varphi_\p = C }} 1 \sim \frac{|C|}{|G|} {\rm Li}(x), $$
where  ${\rm Li}(x):=\int_{2}^{x} {\d u}/{\log u}$, $\varphi_\p$ (resp. $\mathcal N\mathfrak p$) is the Frobenius at (resp. the norm of) the prime ideal $\p$ (see \emph{e.g.}~\cite{Mar}*{\S4} for the general definition of the Frobenius substitution).
Equivalently, if $t\colon G \rightarrow \mathbb R$ is a real-valued class function, then
$$ \pi(x;L/K,t) := \sum_{\substack{ \p \triangleleft \mathcal O_K \\ \mathcal N \p \leq x }} t(\varphi_\p) \sim \widehat t(1) {\rm Li}(x), $$
where $\widehat t(1) = \frac 1{|G|}\sum_{g\in G} t(g)$.
Note that if ${\bf 1}_C$ denotes the indicator function of a given conjugacy class $C$ of $G$, then $\pi(x;L/K,{\bf 1}_C)=\pi_C(x;L/K)$.
As for the error term, which was first bounded effectively by Lagarias and Odlyzko~\cite{LO}, Bella\"iche~\cite{Be} has shown under GRH and Artin's conjecture, that in the case $K=\Q$ and for $x\geq 3$,
$$ \pi(x;L/K,t)-\widehat t(1){\rm Li}(x) \ll \lambda_{1,1}(t) \sqrt{x}\log(xM|G|), $$
where $M$ is the product of all primes ramified in $L$ and $\lambda_{1,1}(t) := \sum_{\chi \in \Irr(G)} \chi(1)|\widehat t(\chi)|,  $ with $\Irr(G)$ being the set of irreducible characters of $G$ and $$ \widehat t(\chi):= \langle t,\chi\rangle_G = \frac 1{|G|}\sum_{g\in G} \overline{\chi (g)}  {t(g)}\,.$$ As an example, if $t={\bf 1}_C$ for some conjugacy class $C\subset G,$ then $\widehat t(\chi)=\frac{|C|}{|G|}\overline{\chi(C)}.$

Bella\"iche's bound has been generalized and improved in the recent work~\cite{FJ}. Moreover, \emph{loc. cit.}
studies the generic behaviour of the error term, in particular its limiting distribution as~$x\rightarrow \infty$. Using probabilistic tools, 
a sufficient condition is obtained for this error term to be Gaussian~\cite{FJ}*{Proposition 5.8}. This generalizes previous work~\cite{H7}, ~\cite{RS} and~\cite{FM} on primes in arithmetic progressions. For example, Hooley has shown that for $(a,q)=1$, the error term
$$ E(x;q,a):= \sum_{\substack{n\leq x \\ n\equiv a \bmod q}} \Lambda(n) - \frac 1{\phi(q)}\sum_{\substack{n\leq x}} \Lambda(n)$$
is such that for any fixed $r\in \mathbb N$, 
$$ \lim_{q\rightarrow \infty}  \lim_{X\rightarrow\infty} \frac{\phi(q)^{\frac r2}}{(\log q)^{\frac r2}} \frac 1{\log X} \int_2^{X} \frac{(E(x;q,a))^r}{x^{\frac r2}} \frac{\d x}x = \mu_r, $$
where
$$\mu_{r} := \begin{cases} (2n-1) \cdot (2n-3) \cdots 1 & \text{ if } r=2n,\\
0 & \text{ otherwise,}
\end{cases}$$
is the $r$-th moment of the Gaussian. Hooley's theorem is conditional on the GRH, as well as the assumption that the multiset of non-negative non-trivial zeros of Dirichlet $L$-functions modulo $q$ is linearly independent over the rationals.

The results which we just described (including a number of results in~\cite{FJ}) apply to limiting distributions as $x\rightarrow \infty$, and thus do not give information on the behaviour of the error term uniformly when $q$ varies with $x$. In fact, to obtain such explicit information one would need to significantly strengthen the linear independence hypothesis, that is one would need to assume that integer linear combinations of $L$-function zeros are bounded away from zero as a function of $q$ (in the spirit of~\cite{MV07}*{\S 15.3}).

In the recent paper~\cite{BF2},
a lower bound is established on higher moments of primes in progressions in a certain range of $q$ in terms of $x$, assuming only the GRH. More precisely,
the results of \emph{loc. cit.} manage to circumvent the linear independence assumption by considering a weighted version of $E(x;q,1)$ and applying positivity in the explicit formula.

The goal of the present paper is to generalize these results in the context of the Chebotarev density theorem, that is to obtain lower bounds on moments of a weighted version of the error term $\pi(x;L/K,t) - \widehat t(1) {\rm Li}(x)$ in certain ranges of $x$ depending on the class function $t$ and on invariants of the extension $L/K$ such as the size of its Galois group and of the root discriminant of $L$. We stress that our results do not assume any form of linear independence of the $L$-function zeros involved.

Before we state our results, we need a few definitions. We let $\delta>0$ and $\mathcal S_\delta\subset \mathcal L^1(\mathbb R)$ be the set of all non-trivial differentiable even $ \eta :\R \rightarrow \R$ such that for all $t\in \R$,
$$\eta (t) ,  \eta'(t) \ll \e^{-(\frac 12 +\delta) |t|}, $$
and moreover for all $\xi \in \R$, we have that\footnote{The upper bound on $\widehat \eta(\xi)$ is a quite mild condition given the differentiability of $\eta$; going through the proof of the Riemann-Lebesgue lemma
we see for instance that a stronger bound holds as soon as $\eta'$ is monotonous. (A stronger bound holds if $\eta$ is twice differentiable.) As for the positivity condition, we can take for example $\eta = \eta_1 \star \eta_1$ for some smooth and rapidly decaying $\eta_1$.}
 \begin{equation}
 0 \leq \widehat \eta(\xi) \ll (|\xi|+1)^{-1} (\log(|\xi|+2))^{-2-\delta}.
 \end{equation}
 Here, the Fourier transform is defined by
 $$  \widehat \eta(\xi) := \int_{\mathbb R} \e^{-2\pi i \xi u} \eta(u) \d u\,.$$
 Finally for any $h\in\mathcal L^1(\mathbb R)$ we define
 $$
 \alpha(h):=\int_{\R} h(t)\,{\rm d}t\,.
 $$
In this notation, one of the goals of the paper~\cite{BF2} is to give lower bounds on moments of the error term
\begin{equation}  \sum_{\substack{ n\geq 1   \\ n\equiv 1 \bmod q }}\frac{\Lambda(n) }{n^{\frac 12}} \eta\big(\log (n/x) \big)- \frac 1{\phi(q)}\sum_{\substack{ n\geq 1   \\ (n, q)=1 }}\frac{\Lambda(n) }{n^{\frac 12}} \eta\big(\log (n/x) \big) ,
 \label{eq:BF2}
\end{equation}
which is a weighted version of $\psi(x;q,1)-\frac 1{\phi(q)} \psi(x,\chi_{0,q})$ where $\chi_{0,q}$ is the principal character modulo $q$.
 
In this paper we consider $L/K$ a Galois extension of number fields of group $G=\Gal(L/K)$ and we fix a real-valued class function\footnote{We will require later the condition $\widehat t\geq 0$. In particular, the results of our paper also apply to class functions of the form ${\rm Re}(t)$, where $t\colon G \rightarrow \mathbb C$ is a class function of non-negative real part such that $\widehat{t}\geq 0$.} $t\colon G \rightarrow \mathbb R$. Our goal will be to understand the moments of
\begin{equation}\label{eq:psi}
 \psi_\eta(x;L/K,t):=\sum_{\substack{\mathfrak p\triangleleft \mathcal O_K \\ m\geq 1  }} t(\varphi_\mathfrak p^m) \frac{\log (\mathcal N\mathfrak p)}{\mathcal N \mathfrak p^{\frac m2}} \eta\big(\log (\mathcal N\mathfrak p^m/x)\big)\,,
 \end{equation}
which is a direct generalization of~\eqref{eq:BF2} (where, disregarding ramified primes, $K=\Q$, $L=\Q(\zeta_q)$ and $t={\bf 1}_{1\bmod q} - \frac 1{\phi(q)}$).
First, we notice that with this smooth weight, the Chebotarev density theorem reads
$$ \psi_\eta(x;L/K,t) \sim  \widehat t(1)  x^{\frac 12} \mathcal L_\eta(\tfrac 12),$$
where
$$ \mathcal L_{\eta}(u):= \int_{\mathbb R} \e^{ux} \eta(x)\d x$$
(note that $\mathcal L_{\eta}(u)=\mathcal L_{\eta}(-u)$).
Secondly, it follows from an analysis as in~\cite{FJ} (see \emph{e.g.}~\cite{FJ}*{Th.~2.1}) that under GRH, the remainder term $ \psi_\eta(x;L/K,t) - \widehat t(1) x^{\frac 12} \mathcal L_\eta(\tfrac 12)$ has average value equal to $\widehat \eta(0)z(L/K,t)$, where we define
$$ z(L/K,t):= \sum_{\chi \in \Irr(G)}\widehat t(\chi)\ord_{s=\frac 12}L(s,L/K,\chi).$$

With this in mind, we define $\mathcal U$ to be the set of even non-trivial integrable functions $\Phi \colon \R \rightarrow \R $ such that\footnote{Note that those conditions imply that $\widehat \Phi(0) >0$.} $\Phi,\widehat \Phi \geq 0$, and we consider for $U>0$, $\Phi\in\mathcal U$, $n\in\Z_{\geq 1}$, and $\eta\in\mathcal S_\delta$ the central moment
\begin{equation}
\label{equation definition Mtilde}
\begin{split}
\widetilde M_n(U,&L/K;t ,\eta,\Phi)\\ &:=  \frac{1}{ U\int_0^\infty \Phi }  \int_0^\infty \Phi\big(\tfrac uU\big) \big(\psi_\eta(\e^u;L/K,t)- \widehat t(1) \e^{\frac u2} \mathcal L_\eta(\tfrac 12)  - \widehat \eta(0)z(L/K,t) \big)^n {\rm d} u\,.
\end{split}\end{equation} 
We will see that under GRH and Artin's conjecture (denoted AC throughout the paper; see~\S\ref{section:Artin} for recollections on Artin $L$-functions), this integral converges.

Our main result is a lower bound on the even moments $\widetilde M_{2 m}(U;L/K;t ,\eta,\Phi)$, which is conditional on GRH as well as Artin's conjecture. More precisely, if Artin's conjecture holds for a Galois extension $L/F$ where $K/F$ is a finite extension, then we obtain a bound which depends on $F$. For simplicity one can assume that $F=\Q$; in general, we expect to obtain the best possible (and in many families asymptotically optimal) bound with this choice. Our bounds will depend on the root discriminant
\begin{equation}\label{eq:rootdisc}
{\rm rd}_L := d_{L}^{\frac 1{[L:\Q]}}\,,
\end{equation}
where $d_{L}$ is the absolute value of the discriminant of $L/\Q$. Our estimates will also involve various norms relative to the Galois groups $G$ and $G^+$ of the extensions $L/K$ and $L/F$ respectively. For a finite group $\mathcal G$ and for a class function $t\colon \mathcal G\to\C$, these norms are defined as follows: 
$$\lambda_{j,k}(t):=  \sum_{\chi \in \Irr(\mathcal G)} \chi(1)^j|\widehat t(\chi)|^k\,,\qquad (j,k\geq 0)\,.    $$

Our main results (Theorem~\ref{theorem main} and Theorem~\ref{theorem main 2}) show that the moments
we study $\widetilde M_{2m}(U,L/K;t,\eta,\Phi) $
are asymptotically greater than or equal those of a Gaussian of expected variance.
 The implied variance will be expressed in terms of zeros of Artin $L$-functions of a Galois number field extension $L/F$.
More precisely, denoting $t^+:={\rm Ind}_{G}^{G^+} t$, this variance takes the shape
\begin{equation}\label{defVLK}
\nu(L/F,t^+;\eta):= \sum_{ \chi \in \Irr(G^+)}  |\widehat t^+(\chi)|^2b_0(\chi;\widehat \eta^2),
\end{equation}
and moreover for $\chi\in\Irr(\Gal(L/F))$,
\begin{equation}
\label{def:b0eta}
 b_0(\chi;\widehat \eta^2):=\sum_{\rho_\chi\notin \mathbb R} \Big| \widehat \eta \Big(\frac{\rho_\chi-\frac 12}{2\pi i }\Big) \Big|^2,
\end{equation}
where $\rho_\chi$ is running over the non-trivial zeros of $L(s,L/F,\chi)$.


\begin{theorem}
\label{theorem main}
Let $L/K/F$ be a tower of number fields such that $L\neq \Q$, $L/F$ is Galois, and assume GRH and AC for the extension~\footnote{Note that AC for the extension $L/F$ implies AC for the extension $L/K$.}  $L/F$. Define $G:=\Gal(L/K)$, $G^+:=\Gal(L/F)$, let $\eta \in \mathcal S_\delta$, $\Phi \in \mathcal U$, and assume that $t \colon G \rightarrow \R$
is a non-zero class function such that $t^+:={\rm Ind}_{G}^{G^+} t$, the class function on $G^+$ induced by $t$, satisfies\footnote{See the beginning of~\S4 for recollections on induction. Notice that the condition $\widehat t^+ \geq 0$ is weaker than  $\widehat t \geq 0$. Indeed, by Frobenius reciprocity, we have that $\widehat {t^+}(\chi)= \widehat t(\chi|_G) $, and moreover the character $\chi|_G$ is a sum of irreducible characters of $G$.} $\widehat {t^+}  \in \mathbb R_{\geq 0} $.  For $m\in \mathbb N$, we have the lower bound

\begin{multline}
 \widetilde M_{2m}(U,L/K;t,\eta,\Phi) \geq \mu_{2m} \nu(L/F,t^+;\eta)^m \Big( 1 + O_\eta\big(m^2m! w_4(L/F,t^+;\eta) \big)\Big)\\
+O\Big( \frac {(\kappa_\eta [F:\Q]\, \lambda_{1,1}(t^+) \log({\rm rd}_L))^{2m}}U \Big),
\label{equation theorem main}
\end{multline}
where $\kappa_\eta>0$ is a constant which depends only on $\eta$ and
\begin{equation}\label{defVLP}
w_4(L/F,t^+;\eta):= \frac{\sum_{ \chi \in \Irr(G^+)}  |\widehat t^+(\chi)|^4b_0(\chi;\widehat \eta^2)}{\Big(\sum_{ \chi \in \Irr(G^+)}  |\widehat t^+(\chi)|^2b_0(\chi;\widehat \eta^2) \Big)^2}.
\end{equation}

\end{theorem}
In other words, the moments $\widetilde M_{2m}(U,L/K;t,\eta,\Phi) $ are at least Gaussian of variance equal to $\nu(L/F,t^+;\eta)$. Our next main result is an estimation of this variance as well as an upper bound on the error term $w_4(L/F,t^+;\eta)$.

\begin{remark}
A version of the quantity $w_4(L/F,t^+;\eta)$ has already appeared in the probabilistic study of the error term in Chebotarev~\cite{FJ}*{\S 5.2}. In particular, the condition $w_4(L/F,t^+;\eta)=o(1)$ was necessary in order to obtain the central limit 
theorem~\cite{FJ}*{Proposition 5.8}. However, there exists class functions for which this condition does not hold: taking for instance $t=1$, we obtain a weighted version of the error term in the prime number theorem which under standard hypotheses is not Gaussian (this goes back to Wintner~\cite{Wi}). Another instance of non-Gaussian moments is explored in~\cite{BFJ2}.  
\end{remark}

In order to state our bounds on the variance $\nu(L/F,t^+;\eta)$, we define the following quantity attached to a non-trivial class function\footnote{Note that if $G=\{1\}$, then we define $S_t:=0$.} $t\colon G\to \R$

\begin{equation}
\label{def St}
 S_t := \max_{1\neq a\in G}\frac{\Big|\sum_{\chi \in \Irr(G)} \chi(a)  |\widehat t(\chi)|^2\Big|}{\sum_{\chi \in \Irr(G)} \chi(1)|\widehat t(\chi)|^2} =
 \max_{1\neq a\in G}\frac{\Big|\sum_{\chi \in \Irr(G)} \chi(a)  |\widehat t(\chi)|^2\Big|}{\lambda_{1,2}(t)}\leq 1.
 \end{equation}

\begin{remark}\label{rem:Stabelian}
The quantity $S_t$ is, in a sense, a measure of the size of the support of $ \widehat t$. For many groups, we expect $S_t$ to be much smaller than $1$ as soon as $ \widehat t$ has a ``large" support in~$\Irr(G)$ (see the example following Theorem~\ref{theorem main 2} as well as \S\ref{section examples}).  
\end{remark}

Here and throughout we denote by $\log_k$ the $k$-fold iterated logarithm.

\begin{theorem}
\label{theorem main 2}
With the same notations and assumptions as in Theorem~\ref{theorem main}, we have the following.

\begin{itemize}

 \item Assume that the weight function $\eta$ is such that\footnote{More generally, it is sufficient to assume that there exists an interval $[T_1,T_2]$ where $T_1>\kappa$ and $T_2-T_1 \geq \kappa (\log_2(T_1))^{-1}$ on which $ \widehat \eta$ does not vanish, where $\kappa>0$ is a large enough absolute constant.} $\inf\{  |z- z'| : z\neq z', \widehat \eta(z)=\widehat \eta(z')=0 \}>0$. Then,
 we have the bounds
$$ \nu(L/F,t^+;\eta)  \asymp_\eta \sum_{ \chi \in \Irr(G^+)} |\widehat {t^+}(\chi)|^2 \log (A(\chi){+2} ),$$
$$ w_4(L/F,t^+,\eta) \ll_\eta \frac{\sum_{ \chi \in \Irr(G^+)} |\widehat t^+(\chi)|^4 \log(A(\chi)+2)}{\Big(\sum_{ \chi \in \Irr(G^+)} |\widehat t^+(\chi)|^2 \log(A(\chi)+2) \Big)^2}.  $$
Here, $A(\chi)$ is the Artin conductor which is defined in~\eqref{eq:defA}.

\item Assume that $ S_{t^+} \leq 1- \kappa_\eta \big(\log_2 (\rd_L+2)\big)^{-1}$ where $\kappa_\eta>0$ is a large enough  constant which depends only on $\eta$. Then we have the more explicit bounds 
 \begin{equation}\begin{split}
 1-S_{t^+}-O_\eta\Big(\frac 1{\log_2 (\rd_L+2)}\Big)  \leq &\frac{\nu(L/F,t^+;\eta)}{\alpha(|\widehat \eta |^2) [F:\Q] \log(\rd_L)  \lambda_{1,2}(t^+)}  \\
&\qquad\qquad \leq  1+S_{t^+}+O_\eta\Big(\frac 1{\log_2 (\rd_L+2)}\Big),
\end{split} \label{equation lower bound on nu}
 \end{equation}
as well as\footnote{Note that the second bound here shows that $w_4(L/F,t^+;\eta)$ is small as soon as the root discriminant is large. However, this bound is far from optimal, and we expect the quotient $\lambda_{1,4}(t^+)/\lambda_{1,2}(t^+)^2$ to also be small in many cases.} 
\begin{align*}
 w_4(L/F,t^+;\eta) [F:\Q]\log (\rd_L) &\ll_\eta \frac{ \lambda_{1,4}(t^+) }{ \lambda_{1,2}(t^+)^2    } \Big(1-S_{t^+}-O_\eta\Big( \frac 1{\log_2(\rd_L+2)} \Big)\!\Big)^{-2} \!\!\ll_\eta  {   (\log_2  \rd_L)^2}  .
\end{align*}

\end{itemize}

\end{theorem}

\begin{remark}

To see why the assumptions made in Theorem~\ref{theorem main 2} are important, consider the case where $K=\Q$ and $t=t^+=1$, in which
$S_{t^+}=1$. Then we have that
$$ \psi_\eta(x;L/K,t)- x^{\frac 12} \mathcal L_\eta(\tfrac 12)=\sum_{\substack{p  \\ m\geq 1  }} \frac{\log p}{ p^{\frac m2}} \eta\big(\log (p^m/x)\big) -x^{\frac 12} \mathcal L_\eta\big(\tfrac 12\big), $$
and the moments of the limiting distribution of this function are much smaller than those of a Gaussian (in fact the limiting distribution has compact and uniformly bounded support, which does not depend on the extension $L/K$). This does not contradict Theorem~\ref{theorem main}, since in this case $w_4(L/F,t^+,\eta) \gg 1 $ (hence we cannot extract any information from~\eqref{equation theorem main}).

\end{remark}


\begin{remark}
The norms $\lambda_{j,k}(t):=  \sum_{\chi \in \Irr(G)} \chi(1)^j|\widehat t(\chi)|^k    $ play a fundamental role in the analysis of the error term in the Chebotarev density theorem. Bella\"{\i}che~\cite{Be} coined the term ``Littlewood norm" for $\lambda_{1,1}(t)$, which he thoroughly studied with applications to the sup norm of the error term in Chebotarev. The norm $\lambda_{1,2}(t)$ and its applications to the mean square of the error term in Chebotarev were studied in~\cite{FJ}.

\end{remark}

\begin{remark}
\label{remark character removal}
One can generalize the bound~\eqref{equation lower bound on nu}. If $\Xi \subset \Irr(G^+)$ is a set of irreducible characters, then on can drop the terms where $\chi \notin \Xi$ in the definition~\eqref{defVLK} of $\nu(L/F,t^+;\eta)$. Doing so, and assuming that  
$$ S_{t^+}(\Xi) := \max_{1\neq a\in G}\frac{\Big|\sum_{\chi \in \Xi} \chi(a)  |\widehat t(\chi)|^2\Big|}{\sum_{\chi \in \Xi} \chi(1)|\widehat t(\chi)|^2}  \leq 1- \kappa_\eta (\log_2 (\rd_L+2))^{-1},$$ where $\kappa_\eta>0$ is a large enough  constant which depends only on $\eta$, we deduce the bound
 \begin{equation*}
 \frac{\nu(L/F,t^+;\eta)}{\alpha(|\widehat \eta |^2) [F:\Q] \log(\rd_L)  \lambda_{1,2}(t^+;\Xi)} \geq  1-S_{t^+}(\Xi)-O_\eta\Big(\frac 1{\log_2 (\rd_L+2)}\Big)
 \end{equation*}
 where  $ \lambda_{1,2}(t^+;\Xi):=  \sum_{\chi \in \Xi } \chi(1)|\widehat t(\chi)|^2.$    
This generalized bound will be useful in the case $G^+=S_n$ (see \S~\ref{section S_n}).
\end{remark}

The following example illustrates the relevance of introducing the quantities $S_t$ and $S_{t^+}$ in the statement of Theorem~\ref{theorem main 2}.

\begin{example}


Fix an abelian extension $L/K$ of number fields and let $G=\Gal(L/K)$. Let $t$ be real valued with non-negative Fourier coefficients of constant modulus (\emph{e.g.} $t={\bf 1}_g$, for any~$g\in G$), then, since by orthogonality $\sum_{\chi\in \Irr(G)}\chi(a)=0$ for every $a\in G\smallsetminus\{1\}$, we have~$S_t=0$. In particular~\eqref{equation lower bound on nu} combined with~\eqref{equation theorem main} generalizes the situation considered in~\cite{BF1}*{page
~7} where $t={\bf 1}_{1\mod q}$ and $G\simeq (\Z/q\Z)^\times$ is the Galois group of the cyclotomic extension $\Q(\zeta_q)/\Q$. For further examples, including non abelian extensions, see \S\ref{section examples}.

\end{example}

\begin{remark}
In Theorem~\ref{theorem main}, one might wonder whether it is possible to bound the more familiar moments
$$
 M_n(U,L/K;t,\eta,\Phi):=  \frac{1}{ U\int_0^\infty \Phi } \int_0^\infty \Phi\big(\tfrac uU\big)  \big(\psi_\eta(\e^u;L/K,t)- \widehat t(1) \e^{\frac u2} \mathcal L_\eta(\tfrac 12) \big)^n {\rm d} u,
$$
rather than $\widetilde M_n(U,L/K;t,\eta,\Phi)$.
This is indeed the case since in Theorem~\ref{theorem main},
$$ m_{L/K;t,\eta}:= \widehat \eta(0)z(L/K,t)
= \widehat \eta(0)z(L/F,t^+)\,
$$ (this follows from \cite{FJ}*{Lemma 3.15}), which by our assumptions is non-negative. Then, we have that
$$M_{2m}(U,L/K;t,\eta,\Phi) = \sum_{j=0}^{2m} \binom {2m}j \widetilde M_j(U;L/K,t)m_{L/K;t,\eta}^{2m-j} \geq \widetilde M_{2m}(U,L/K;t,\eta,\Phi). $$
Of course, if we can show that $m_{L/K;t,\eta}>0$, then the last bound can be improved. As a result, we obtain the following corollary.
\end{remark}
\begin{corollary}
\label{corollary other moment}
Under the same assumptions as in Theorem~\ref{theorem main}, the bound~\eqref{equation theorem main} holds with $M_{2m}(U,L/K;t,\eta,\Phi)$ in place of $\widetilde M_{2m}(U,L/K;t,\eta,\Phi)$.
\end{corollary}

We end this section by noting that Theorems~\ref{theorem main}  and~\ref{theorem main 2} imply $\Omega$-results on the classical (unweighted) prime ideal counting functions
\begin{equation}\label{eq:psiunweighted}
 \psi(x;L/K,t):=\sum_{\substack{\mathfrak p\triangleleft \mathcal O_K \\ m\geq 1  }} t(\varphi_\mathfrak p^m) \log (\mathcal N\mathfrak p)\,.
 \end{equation}

 \begin{corollary}
 \label{corollary unweighted}
 Let $L/K$ be a Galois extension of number fields for which GRH holds. Let $F$ be any subfield of $K$ (\emph{i.e.} $F \subset K \subset L$ ) which is such that $L/F$ is Galois and satisfies AC. 
Define $G:=\Gal(L/K)$, $G^+:=\Gal(L/F)$, and assume that $t \colon G \rightarrow \R$
is a non-zero class function such that $t^+:={\rm Ind}_{G}^{G^+} t$ satisfies $\widehat {t^+}  \in \mathbb R_{\geq 0} $. Assume that $ S_{t^+} \leq 1- \kappa (\log_2 (\rd_L+2))^{-1}$ where $\kappa>0$ is a large enough absolute constant. Then there exists a sequence of values $x=x_{j;L/K,t}$ tending to infinity such that
\begin{equation}
\Big| \psi(x;L/K,t) -  \widehat t(1)  x \Big| \gg x^{\frac 12}\Big(  [F:\Q] \log(\rd_L) \lambda_{1,2}(t^+)\Big)^{\frac 12} \Big(1-S_{t^+}- O\Big( \frac 1{\log_2(\rd_L+2)} \Big) \Big)^{\frac 12},
\label{equation corollary oscillation}
\end{equation}
where the implied constant is absolute.
 More precisely, there exists a large enough absolute constant $\kappa'>0$ such that for any large enough $U >0$ (in absolute terms), there exists $x>1$ such that~\eqref{equation corollary oscillation} holds with $\log x\in [U, U \cdot \beta_{L,F,K,t}]$ where
 $$ \beta_{L,F,K,t}:= \kappa'[F:\Q] \lambda_{1,1}(t^+)^2\log(\rd_L+2)\log_2(\rd_L+2)  / \lambda_{1,2}(t^+)\,. $$
 \end{corollary}

 \begin{corollary}
 \label{corollary identity}
 Let $L/K$ with $L\neq \Q$ be a Galois extension of number fields for which GRH holds, and define $G:=\Gal(L/K)$. Then for any large enough $U >0$, there exists $x>1$ for which $\log x\in [U,\kappa' U \cdot \log(d_L+2)]$ and
 such that
 \begin{equation}
\big| \psi(x;L/K,|G|{\bf 1}_{e}) -  x \big| \gg x^{\frac 12}\big( \log d_L \big)^{\frac 12} .
\end{equation}
Here, $\kappa' $ is a large enough absolute constant (in absolute terms).

 \end{corollary}

The paper is organized as follows. In~\S\ref{section examples} we state applications of our main results to specific families of Galois extensions of number fields. The proofs of these statements are postponed to~\S\ref{section:prooftables}. Next,~\S\ref{section:cond} and~\S\ref{section:Artin} are dedicated to recollections and preparatory results concerning Artin conductors, and zeros of Artin $L$-functions, respectively. We prove our main results as well as Corollary~\ref{corollary unweighted} and Corollary~\ref{corollary identity} in~\S\ref{section:artin}.





\section{Explicit Families of Galois extensions and class functions}
\label{section examples}

In this section we study explicit infinite families of extensions for which Theorems~\ref{theorem main} and~\ref{theorem main 2} apply. The proofs of these results are contained in~\S\ref{section:prooftables}.

\subsection{Dihedral extensions}

A natural next step after analyzing the abelian case (see Remark~\ref{rem:Stabelian}) is to consider groups having an abelian subgroup of small index. Such is the case of dihedral groups. Let us start by recalling classical facts (see \emph{e.g.}~\cite{Se}*{\S 5.3}): for an odd integer $n\geq 3$, the dihedral group of order $2n$ is defined as follows,
$$
D_n=\langle \sigma,\tau\colon \sigma^n=\tau^2=1,\, \tau\sigma\tau=\sigma^{-1}\rangle\,.
$$
The nontrivial conjugacy classes of $D_n$ are
$$
\{\sigma^j,\sigma^{-j}\}\,\, (1\leq j\leq \tfrac12(n-1) )\,, \qquad \text{and}\qquad
\{\tau\sigma^k\colon 0\leq k\leq n-1\}\,.
$$

\begin{proposition}\label{prop:tableDn}

 One has the following table of values of $S_t$ for various choices of central functions $t\colon D_n\to \R$.

\begin{center}
        \begin{tabular}{ |c | c | c | c | }
        \hline
            $n$  & $\geq 3$ & $\geq 3$ & 
            $\geq 5$ \\
        \hline
            $t$  & $|D_n|{\bf 1}_e$ & ${\bf 1}_{\{\sigma,\sigma^{-1}\}}$ & 
            $2{\bf 1}_e+{\bf 1}_{\{\sigma,\sigma^{-1}\}}$ \\
        \hline
            $S_t$ & $\tfrac 1{2n-1}$ & $\tfrac{1-2/n}{2(1-1/n)}$ & $< \tfrac 23 $\\
        \hline
        \end{tabular}\,.
    \end{center}
  \end{proposition}

The first column of the table is used to prove the following result.

\begin{proposition}\label{prop:dihedral}
For $n\geq 3$ odd, let $L/\Q$ be a $D_n$-extension of number fields for which GRH holds. One has for any $m\geq 1$, any $\eta\in\mathcal S_\delta$ and any $\Phi\in\mathcal U$,
 \[
\widetilde M_{2m}(U,L/\Q;|D_n|{\bf 1}_e,\eta,\Phi) \geq \mu_{2m}\Big(\alpha(|\widehat{\eta}|^2)\Big(2-\frac 1n\Big)\log d_L\Big)^m \big(1+o_{\rd_L\to\infty}(1)\big)\,,
\]
as soon as $  (\log d_L)^m =o_{d_L\to\infty}(U)$.
\end{proposition}



\subsection{Radical extensions}\label{subsec:radical}

We consider the following Galois extension studied in~\cite{FJ}*{\S 9.2}. Let $a,p$ be distinct prime numbers such that $p\neq 2$ and $a^{p-1}\not\equiv 1\bmod p^2$
and let $K_{a,p}$ be the splitting field (inside $\C$) of $X^p-a\in\Q[X]$. The Galois group $G:=\Gal(K_{a,p}/\Q)$ is isomorphic to the group of affine transformations of $\mathbb{A}^1_{\F_p}$. A convenient way to describe $G$ is the following:

\begin{equation}\label{eq:groupradical}
G\simeq  \left\{\left(\begin{array}{cc}
c & d \\
0 & 1
\end{array} \right)\colon c\in \F_p^*,\, d\in\F_p\right\}\,.
\end{equation}
One has $|G|=p(p-1)$ and $G$ admits a real irreducible character $\vartheta$ of degree $p-1$ (see~\S\ref{subsec:rad}).

\begin{proposition}\label{prop:tableradical}
Let $G$ be as in~\eqref{eq:groupradical}. One has the following table of values for $S_t$ for various choices of central functions $t\colon G\to \R$. 

\begin{center}
        \begin{tabular}{ |c | c | c |  }
        \hline
            $t$  & $|G|{\bf 1}_e$ & $\vartheta$  \\
        \hline
            $S_t$ & $\tfrac 1{p(1- 2/p+ 2/p^2)}$ & $\tfrac 1{p-1}$ \\
        \hline
        \end{tabular}\,.
    \end{center}
\end{proposition}

We deduce the following result on the moments attached to the class functions considered in the table of Proposition~\ref{prop:tableradical}.

\begin{proposition}\label{prop:radical}
 Let $a,p$ be distinct prime numbers such that $p\neq 2$ and $a^{p-1}\not\equiv 1\bmod p^2$. Let $K_{a,p}/\Q$ be the Galois extension of group $G$ defined by~\eqref{eq:groupradical}. Assuming that GRH holds for $K_{a,p}$, one has for any $m\geq 1$, any $\eta\in\mathcal S_\delta$ and any $\Phi\in\mathcal U$,
 \begin{align*}  
\widetilde M_{2m}(U,K_{a,p}/\Q;|G|{\bf 1}_e,\eta,\Phi) &\geq \mu_{2m}\big(\alpha(|\hat{\eta}|^2) p^3\log p\big)^m \big(1+o_{p\to\infty}(1)\big)\,,\cr
\widetilde M_{2m}(U,K_{a,p}/\Q;\vartheta,\eta,\Phi) &\geq \mu_{2m}\big(\alpha(|\hat{\eta}|^2) p\log p\big)^m \big(1+o_{p\to\infty}(1)\big)\,,
 \end{align*}
as soon as $  (p\log p)^m =o_{p\to\infty}(U)$.


\end{proposition}



Note that in this particular example of Galois extension $K_{a,p}/\Q$ the Artin conductors of the elements of $\Irr(G)$ can be explicitly computed (see~\cite{FJ}*{\S 9.2} and~\cite{V}), therefore the last estimates of Theorem~\ref{theorem main 2} can also be applied (yielding a weaker bound). Specific features of moments in the Chebotarev density Theorem for Galois extensions of type generalizing the case of $K_{a,p}/\Q$ are studied in detail in~\cite{BFJ2}.

\subsection{Moments for irreducible characters}

As already mentioned in Remark~\ref{rem:Stabelian}, choosing~$t$ such that $\widehat{t}(\chi)=0$ for many irreducible characters $\chi$ of $G$ leads in general to a value of $S_t$ that is very close to $1$ (cancellations in character sums are believed to occur only when the sums are taken over a sufficiently large index set). However, in some cases where $t$ is non trivial but has a Fourier support of minimal size (\emph{e.g.} when $t$ is a non trivial irreducible character of $G$, as in the case of $t=\vartheta$ in~\S\ref{subsec:radical}), one can still have $S_t<1$ so that our main estimates in Theorem~\ref{theorem main} and~\ref{theorem main 2} apply. The following statement gives a setup where one can take $t$ to be very close to an irreducible character and still apply our main results. This result covers the situation lying at the opposite of the generalization of the bound~\eqref{equation lower bound on nu} discussed in Remark~\ref{remark character removal}, where one discriminates the irreducible characters appearing in the Fourier support of the class function $t$ according to the size of their degree.


\begin{proposition}\label{prop:uniquechi}
 Let $L/K/F$ be a tower of number fields such that $L\neq \Q$, $L/F$ is Galois, and assume GRH and AC for the extension $L/F$. Define $G:=\Gal(L/K)$ and $G^+:=\Gal(L/F)$. Let $t \colon G \rightarrow \R$ be a class function such that $t^+=\frac{\chi+\overline{\chi}}2$ for some $\chi\in\Irr(G^+)$. Let~$\rho$ denote an irreducible representation of $G^+$ of character $\chi$. Let $\eta \in \mathcal S_\delta$ and $\Phi \in \mathcal U$. Then $S_{t^+}<1$ if and only if $\rho$ is faithful and $G^+$ has a center of odd order. In particular, if this last condition holds and if ${\rm rd}_L$ is large enough in terms of $1-S_{t^+}$, then~\eqref{equation lower bound on nu} applies.
\end{proposition}

 Finite groups admitting faithful irreducible characters are classified by a result of Gasch\"utz (see \emph{e.g.}~\cite{Hup}*{Th. 42.7}). Finally note that even if $\rho$ is not faithful or $2\mid |Z(G^+)|$ then we may apply the first case in Theorem~\ref{theorem main 2}.



\subsection{$S_n$-extensions}
\label{section S_n}
Perhaps what can be seen as the ``generic'' situation is when $L/\Q$ is Galois of group $S_n$ the symmetric group on $n$ letters. One can obtain explicit lower bounds for $\nu(L/F,t^+;\eta)$ by following the approach in~\cite{FJ}*{\S 7}, which involves Roichman's bound~\cite{Roi} as well as the hook-length formula. For a large set of class functions $t$, one can show that $S_t$ remains bounded away from $1$ (where the distance to $1$ is precisely evaluated as a function of $n$ in \emph{loc. cit.}). For instance this applies to the difference of normalized indicator functions $$t_{C_1,C_2}=(|G|/|C_1|){\bf 1}_{C_1}-(|G|/|C_2|){\bf 1}_{C_2}\qquad ({\rm resp}.\, t_{C}=(|G|/|C|){\bf 1}_{C})$$ as soon as $C_1$, $C_2$ are distinct conjugacy classes of $S_n$, one of which has size at most (resp.~$C$ is a conjugacy class of $S_n$ of size at most) $n!^{1-\frac{4+\eps}{\e\log n}}$.  Using these ideas, we obtain the following result.



\begin{proposition}
\label{proposition S_n}
 Let $n$ be large enough and assume that $L/K$ is a Galois extension of number fields for which $L/\Q$ is Galois of group $S_n$ and satisfies AC and GRH. Let $C_1,C_2$ be conjugacy classes of $\Gal(L/K)$ for which $\min(|C_1^+|,|C_2^+|) \leq n!^{1-\frac{4+\eps}{\e \log n}}$, where $\eps>0$ is fixed. Then for all fixed $m\geq 1$ we have the bound 
\begin{multline*}
    \widetilde M_{2m}(U,L/K;t_{C_1,C_2},\eta,\Phi) \geq \mu_{2m}  \Big( c_\eta \frac{\log (n!/\min(|C_1^+|,|C_2^+|))}{\log n!}  \frac{[K:\Q] \log(\rd_L) n!^{\frac 32}}{\min(|C_1^+|,|C_2^+|)^{\frac 32} p(n)^{\frac 12}} \Big)^m \\ \times ( 1 + o_{\rd_L \rightarrow \infty}(1)),
\end{multline*}
as soon as $([K:\Q]\log(\rd_L) \min(|C_1^+|,|C_2^+|)^3 p(n) /n!^3)^{\frac m2} =o(U)$, where $c_\eta >0$ depends only on $\eta$. The same bound holds for the class function $t_{C_1}=(|G|/|C_1|){\bf 1}_{C_1}$, with the convention that in this case, $\min(|C_1^+|,|C_2^+|)=|C_1^+|$.
\end{proposition}
Note that the factor $\frac{\log (n!/\min(|C_1^+|,|C_2^+|))}{\log n!} \gg_\theta 1 $ as soon as $\min(|C_1^+|,|C_2^+|) \leq n!^{1-\theta}$ for some~$\theta>0$.

\section{Artin conductors}\label{section:cond}

Let us first recall a few facts on Artin conductors.
Consider a finite Galois extension of number fields $L/K$ with Galois group $G$. For $\mathfrak p$ a prime ideal of $\mathcal O_K$ and $\mathfrak P$ a prime ideal of $\mathcal O_L$ lying above
 $\mathfrak p$, the higher ramification groups
 form a sequence
 $(G_i(\mathfrak P/\mathfrak p))_{i\geq 0}$ of subgroups of $G$ (called
 filtration of the inertia group
 ${\rm I}(\mathfrak P/\mathfrak p)$) defined as follows:
$$ G_i(\mathfrak P/\mathfrak p) :=\big\{ \sigma\in G : \forall z \in \mathcal O_L,\, (\sigma z - z) \in
\mathfrak P^{i+1}\big\}. $$
Each $G_i(\mathfrak P/\mathfrak p)$ only depends on $\mathfrak p$ up to conjugation
and $G_0(\mathfrak P/\mathfrak p)={\rm I}(\mathfrak P/\mathfrak p)$ (when conjugation is unimportant we will simply denote this group $I(\mathfrak p)$). For clarity let us
fix prime ideals~$\mathfrak p$ and $\mathfrak P$ as above and write $G_i$ for
$G_i(\mathfrak P/\mathfrak p)$.
 Given a representation $\rho\colon G \rightarrow {\rm GL}(V)$ on a complex vector space $V$,  the subgroups $G_i$ act on $V$ through $\rho$ and we denote by $V^{G_i}\subset V$ the subspace of $G_i$-invariant vectors. Let $\chi$ be the character of $\rho$ and
\begin{equation}
n(\chi,\mathfrak p):=\sum_{i=0}^{\infty} \frac{|G_i|}{|G_0|} {\rm codim} V^{G_i},
\label{eq:definition exponent codimensions}
\end{equation}
which was shown by Artin to be an integer.
The \emph{Artin conductor
of $\chi$} is the ideal of $\mathcal O_K$
$$
\mathfrak f(L/K,\chi):=\prod_\mathfrak p \mathfrak p^{n(\chi,\mathfrak p)}\,.
$$
Note that the set indexing the above product is finite since only finitely many prime ideals~$\mathfrak p$ of $\mathcal O_K$ ramify in $L/K$. We set
\begin{equation}\label{eq:defA}
A(\chi):=d_K^{\chi(1)}\mathcal N_{K/\Q}(\mathfrak f(L/K,\chi))\,,
\end{equation}
where $d_K$ is the absolute value of the absolute discriminant of the number field $K$ and $\mathcal N_{K/\Q}$ is the relative ideal norm with respect to $K/\Q$ (we will use the slight abuse of notation that identifies the value taken by this relative norm map with the positive generator of the corresponding ideal).

We recall the following pointwise bounds on the Artin conductor.
\begin{lemma}[{\cite{FJ}*{Lemma 4.1}}] \label{lemma bounds on Artin conductor}
Let $L/K$ be a finite Galois extension. For any nontrivial irreducible character $\chi$ of $G=\Gal(L/K)$,
one has the bounds
$$ \max\big(1,\tfrac12[K:\Q] \big)\chi(1) \leq \log A(\chi) \leq 2\chi(1) [K:\Q]\log ({\rm rd}_L)\,, $$
where the root discriminant ${\rm rd}_L$ is defined by~\eqref{eq:rootdisc}.
The upper bound is unconditional. The lower bound is unconditional if $K/\Q$ is nontrivial\footnote{It actually also holds for the trivial character in this case.} and holds assuming $L(s,L/\Q,\chi)$ can be extended to an entire function otherwise.
\end{lemma}

We will also use the following average bounds, which generalize~\cite{FJ}*{Lemma 4.2}.
\begin{lemma}\label{lemme:SG}

Let  $L/K$ be a Galois extension of number fields, and let $G=\Gal(L/K)$. Let $\{c_\chi\}_{\chi \in \Irr(G)}$ be a family of non-negative real numbers. Then we have the bounds
$$
(1-S(c))  \sum_{\chi \in \Irr(G)} \chi(1) c_\chi  \leq  \sum_{\chi \in \Irr(G)} \frac{c_\chi \log A(\chi)}{[K:\Q]\log ({\rm rd}_L) } \leq  (1+S(c))    \sum_{\chi \in \Irr(G)} \chi(1)c_\chi,$$
where
$S(c)$ is defined in~\eqref{def St}.

\end{lemma}

\begin{proof}

Denoting by $\chi_{\rm reg}$ the character of the regular representation of $G$, we have the equality
$$  \sum_{\chi \in \Irr(G)} c_\chi \Big( \frac{\chi(1)}{|G|}n(\chi_{{\rm reg}},\mathfrak p)-n(\chi,\mathfrak p) \Big)=\frac{1}{|G_0|} \sum_{i\geq 0} \sum_{1\neq a \in G_i} \sum_{\chi \in \Irr(G)} \chi(a)c_\chi. $$
Summing over the prime ideals  $\mathfrak p$ of $\mathcal O_K$, we deduce that
$$
\Big| \sum_{\chi \in \Irr(G)} \frac{c_\chi \log A(\chi)}{[K:\Q]\log ({\rm rd}_L) }  -  \sum_{\chi \in \Irr(G)} \chi(1)c_\chi   \Big|\\ \leq S(c)  \sum_{\chi \in \Irr(G)} \chi(1)|c_\chi|,
$$
from which the claimed bounds follow.
\end{proof}

We will also use the following bound.
\begin{lemma}
\label{lemme preconducteurs}
Let $L/K$ be a Galois extension of number fields, and let $G=\Gal(L/K)$. For all $\chi \in \Irr(G)$, we have
$$  \frac{\log (A(\chi)+2)}{\log_2 ((A(\chi)+2)^{\frac 3{\chi(1)[K:\Q]}})} \ll [K:\Q]  \chi(1)  \frac{\log ({\rm rd}_L+2)}{\log_2 ({\rm rd}_L+2)}.$$
\end{lemma}
\begin{proof}
This follows form the fact that the function $\cdot / \log \cdot $ is eventually increasing, combined with the upper bound in Lemma~\ref{lemma bounds on Artin conductor}.
\end{proof}

\section{Sums over zeros of Artin $L$-functions}\label{section:Artin}

The goal of this section is to express the function $\psi_\eta(x;L/K,t)$ defined by~\eqref{eq:psi} in terms of a sum over zeros of Artin $L$-functions, which will allow us to give a lower bound on the moments $\widetilde{M}_{2m}(U,L/K;t,\eta,\Phi)$ through an application of positivity. This lower bound will be expressed as a convergent sum over zeros, which we will evaluate explicitly.

First we recall a few facts about Artin $L$-functions. If $\chi$ is an irreducible character of~$G=\Gal(L/K)$ (associated to the representation $\rho$), the corresponding Artin $L$-function is defined for ${\rm Re}(s)>1$ by the Euler product
$$
L(s,L/K,\chi)=\prod_{\substack{\mathfrak p\triangleleft\mathcal O_K\\ \mathfrak p\text{ prime}}} L_{\mathfrak p}(s,\chi)\,,\qquad
\big(L_{\mathfrak p}(s,\chi)=\det\big({\rm Id}-\mathcal N\mathfrak p^{-s} \rho(\varphi_{\mathfrak p}) \big|_{V^{I_{\mathfrak p}}} \big)
\,,\,\, \mathfrak p \triangleleft\mathcal O_K\text{ prime}\big)\,,
$$
where $V^{I_{\mathfrak p}}$ is the subspace of the representation space $V$ which is invariant under the inertia group $I_{\mathfrak p}$ (see~\S\ref{section:cond}).
Artin's conjecture states that $L(s,L/K,\chi)$ can be extended to an entire function (except when $\chi$ is the trivial character, in which case there is a simple pole at $s=1$). Following ~\cite{A}, we recall the definition of the archimedean part $L(s,\chi_\infty)$ of the completed $L$-function associated to the irreducible character $\chi$. Let $v$ be an infinite place of $K$ (that is,~$v$ is a real embedding or a pair of conjugate complex embeddings). Let $w$ be a place of $L$ over
~$v$. For the couple $(w,v)$, the analogue of the decomposition group
is a subgroup $G_{w/v}$ of $G$ which is trivial if $v$ and $w$ are both real or both complex, and which is the group of order two generated by complex conjugation otherwise. If we denote
$$
\Gamma_\R(s)=\pi^{-\frac s2}\Gamma\Big(\frac s2\Big)\,,\qquad
\Gamma_\C(s)=\Gamma_\R(s)\Gamma_\R(s+1)\,,
$$
then the Euler factor at $v$ is
\begin{equation*}
\gamma_v(\chi,s)=
\begin{cases}
\Gamma_\R(s)^{\dim V^{G_{w/v}}}\Gamma_\R(s+1)^{{\rm codim} V^{G_{w/v}}} &\text{ if $v$ is real}\,,\\
\Gamma_\C(s)^{\chi(1)} &\text{ if $v$ is complex}\,.
\end{cases}
\end{equation*}
The archimedean part of the completed $L$-function associated to $\chi$ is then defined by the formula (recall~\eqref{eq:defA})
\begin{equation}\label{eq:arch}
L(s,\chi_\infty)=A(\chi)^{\frac s2}\prod_v\gamma_v(\chi,s)\,.
\end{equation}

We are ready to prove the following explicit formula for the function 
$$ \psi_\eta(x;L/K,\chi):=\sum_{\substack{\mathfrak p\triangleleft \mathcal O_K \\ m\geq 1  }} \chi(\varphi_\mathfrak p^m) \frac{\log (\mathcal N\mathfrak p)}{\mathcal N \mathfrak p^{\frac m2}} \eta\big(\log (\mathcal N\mathfrak p^m/x)\big).$$
\begin{lemma}
\label{lemma explicit formula}
Let $L/K$ be a Galois extension of number fields, denote $G=\Gal(L/K)$, and let $\chi \in \Irr(G)$. Under Artin's conjecture for $L(s,L/K,\chi)$, for any $\eta\in \mathcal S_\delta$ and $x\geq 1$ we have the formula 
$$
\psi_\eta(x;L/K,\chi) =  x^{\frac 12} \mathcal L_\eta\big(\tfrac 12\big) \delta_{\chi=\chi_0}  - \sum_{\rho_\chi} x^{\rho_\chi -\frac 12} \widehat \eta\Big( \frac {\rho_\chi-\frac 12}{2\pi i} \Big)+O_\eta\big(x^{-\frac 12} \log (A(\chi)+2)\big),
$$
where $\rho_\chi$ runs through the non-trivial zeros of $L(s,L/K,\chi)$.

\end{lemma}

\begin{proof}
Let $\gamma_\chi(s)=L(s,\chi_\infty)A(\chi)^{-\frac s2}$. Since we assume Artin's conjecture, we can use~\cite{IK}*{Thm. 5.11} for the test function $\varphi \colon n\mapsto \eta(\log (n/x))/n^{\frac 12}$. Note that our assumptions are weaker than those in~\cite{IK}*{Thm. 5.11}, however going through the proof one sees that our hypotheses are sufficient for~\cite{IK}*{(5.44)} to apply (see \emph{e.g.}~\cite{MV07}*{Th. 12.13} and~\cite{BF2}). Let us recall what is the relevant von Manglodt function $\Lambda_\chi$  in this case (it should satisfy~\cite{IK}*{(5.25)}):
$$
\Lambda_\chi(p^t) =  \sum_{f\ell = t} \sum_{\substack{\mathfrak p \mid p \\ f(\mathfrak p/p) = f}} \log (p^f) \chi(\varphi_{\mathfrak p}^\ell)\,\qquad (p\text{ prime}, t\in \mathbb N)\,.
$$
Indeed, by~\cite{Mar}*{p. 11},
\begin{align*}
-\frac{L'(s,L/K,\chi)}{L(s,L/K,\chi)}& =
\sum_{\substack{\mathfrak p\triangleleft\mathcal O_K\\ \mathfrak p\text{ prime}}} \sum_{\ell \geq 1} \frac{\chi(\varphi_\p^\ell ) \log \mathcal N\mathfrak p }{ \mathcal N\mathfrak p^{s\ell}} \cr&= \sum_{p } \sum_{f,\ell \geq 1}\sum_{\substack{\mathfrak p \mid p \\ f(\p/p)=f }}\frac{\chi(\varphi_\p^\ell ) \log p^f }{ p^{s\ell f}}=
\sum_p\sum_{t\geq 1}\frac{\Lambda_\chi(p^t)}{p^{t s}}\,.
\end{align*}

Then, the first term on the left hand side of~\cite{IK}*{(5.44)} is given by
\begin{align*}
 \sum_{n\geq 1} \Lambda_\chi(n) \frac{\eta(\log (n/x))}{n^{\frac 12}} &= \sum_{p,t}\sum_{f\ell =t} \sum_{\substack{\p\mid p \\ f(\p/p)=f }} \frac{\log(p^f) \chi(\varphi_\p^\ell )\eta(\log (p^t/x))}{p^{\frac{t}2}} \\
&= \sum_{p,t}\sum_{f\ell =t} \sum_{\substack{\p\mid p \\ f(\p/p)=f }} \frac{\log(\mathcal N\p) \chi(\varphi_\p^\ell )\eta(\log (\mathcal N\p^\ell/x))}{\mathcal N \p^{\frac{\ell }2}}\\
&= \sum_{p,\ell}\sum_{t \equiv 0 \bmod \ell} \sum_{\substack{\p\mid p \\ f(\p/p)=t/\ell }} \frac{\log(\mathcal N\p) \chi(\varphi_\p^\ell )\eta(\log (\mathcal N\p^\ell/x))}{\mathcal N \p^{\frac{\ell }2}}\,.
\end{align*}
Reindexing the sums, we obtain
\begin{align*}
\sum_{n\geq 1} \Lambda_\chi(n) \frac{\eta(\log (n/x))}{n^{\frac 12}}&= \sum_{p,\ell}\sum_{t' \geq 1} \sum_{\substack{\p\mid p \\ f(\p/p)=t' }} \frac{\log(\mathcal N\p) \chi(\varphi_\p^\ell )\eta(\log (\mathcal N\p^\ell/x))}{\mathcal N \p^{\frac{\ell }2}}\\
&= \sum_{p,\ell} \sum_{\substack{\p\mid p  }} \frac{\log(\mathcal N\p) \chi(\varphi_\p^\ell )\eta(\log (\mathcal N\p^\ell/x))}{\mathcal N \p^{\frac{\ell }2}} \\
&=  \sum_{\p,\ell}  \frac{\log(\mathcal N\p) \chi(\varphi_\p^\ell )\eta(\log (\mathcal N\p^\ell/x))}{\mathcal N \p^{\frac{\ell }2}}  = \psi_\eta(x;L/K,\chi).
\end{align*}
A similar calculation shows that the second term on the left hand side of~\cite{IK}*{(5.44)} is exactly $\psi_\eta(x^{-1};L/K,\overline{\chi})$. This translates into the formula
\begin{equation}\begin{split}
&\psi_\eta(x;L/K,\chi)+\psi_\eta(x^{-1};L/K,\overline{\chi})=\eta(\log(x))\log A(\chi)+
\delta_{\chi=\chi_0} x^{\frac 12}\mathcal L_\eta\big(\tfrac 12\big) \\ & \, + \frac{1}{2\pi}\int_{-\infty}^\infty\Big(\frac{\gamma_\chi'(\tfrac 12+it)}{\gamma_\chi(\tfrac 12+it)}+\frac{\gamma_\chi'(\tfrac 12-it)}{\gamma_\chi(\tfrac 12-it)}\Big)\widehat{\eta }\big(\tfrac{t}{2\pi}\big) x^{it}\,{\rm d}t
-\sum_{\rho_\chi} x^{\rho_\chi -\frac 12} \widehat \eta\Big( \frac {\rho_\chi-\frac 12}{2\pi i} \Big)+O_\eta(x^{-\frac 12}),
\end{split}
\label{equation explicit formula}
\end{equation}
where the error term accounts for possible trivial zeros of $L(s,L/K,\chi)$ at $s=0$.

To handle the contribution of the integral of $\gamma$-factors we use~\eqref{eq:arch} as well as~\cite{MV07}*{Lemma 12.14} that applies to our case with the choice $J(u)=\eta(2\pi(u-\log x))$. Up to the multiplicative constant $\chi(1)$ the contribution of any infinite place $v$ of $K$ is bounded by an analogous integral where the $\gamma$-factor appearing is the Euler $\Gamma$ function. We can then combine~\cite{MV07}*{Th 12.13, Lemma 12.14} and~\cite{BF2}*{proof of Lemma 2.2} (note that we are using the assumption that $\eta$ is differentiable here). To conclude, we use the upper bound $[K:\Q]\chi(1)\ll \log(A(\chi))$ from Lemma~\ref{lemma bounds on Artin conductor}.
\end{proof}

In Section~\ref{section:artin}, we will apply Lemma~\ref{lemma explicit formula} to approximate
$\widetilde M_n(U,L/K;t ,\eta,\Phi)$ (recall~\eqref{equation definition Mtilde}). A positivity argument will then be applied to this approximation
producing convergent sums over zeros of the form
\begin{equation}
\label{def b}
 b(\chi;h):=\sum_{\rho_\chi}  h\Big(\frac{\rho_\chi-\frac 12}{2\pi i }\Big); \qquad  b_0(\chi;h):=\sum_{\rho_\chi\notin \mathbb R} h\Big(\frac{\rho_\chi-\frac 12}{2\pi i }\Big), \end{equation}
where $\rho_\chi$ runs through the non-trivial zeros of $L(s,L/K,\chi)$.
Note that these sums take into account the multiplicities of zeros, by convention.
As for the involved test function, we will work with $\mathcal T_\delta$, the set of non-trivial measurable functions $h\colon \R \rightarrow \R$ having the following properties.
We require that $\xi\mapsto \xi h(\xi)$ is integrable, and that, for all $\xi \in \R$, we have the bounds
 $$0 \leq h(\xi)\ll (1+|\xi|)^{-1}\big(\log(2+|\xi|)\big)^{-2-2\delta }.$$
Moreover, for all $t\in \R$, we have that\footnote{The integrability of $\xi\mapsto \xi h(\xi)$ implies that $\widehat h$ is differentiable (see \cite{KF89}*{p. 430}).}
 $$\widehat h(t),\widehat h'(t)\ll \e^{-(\frac 12+\frac \delta 2)|t|}.  $$
Note that if $\eta \in  \mathcal S_\delta$ is non-trivial, then $h_\eta:= \widehat \eta^2
\in \mathcal T_\delta$. We may extend $h$ to the domain $\{s\in \mathbb C\colon  |\Im (s)| \leq  \frac {1}{4\pi } \}$ by writing
\begin{equation}
     h(s):= \int_{\mathbb R} \e^{2\pi i s \xi} \widehat h(\xi) \d \xi.
     \label{equation h Fourier}
\end{equation}

\begin{lemma} \label{lemme:explicite2}
Let $L/K$ be a Galois extension of number fields of group $G$, and let $\chi \in \Irr(G)$. Assume Artin's conjecture for the extension $L/K$. Then for any $h\in \mathcal T_\delta$, we have the pointwise estimates
\begin{equation}\label{eq:estimindiv}
b (\chi;h)=\widehat{h}(0)\log {A(\chi)} +O_h\big(\chi(1)[K:\mathbb{Q}]\big);
\end{equation}
$$ b (\chi;h)\ll_h  \log (A(\chi)+2).$$

\end{lemma}

\begin{proof} To estimate the sum $b (\chi;h)$ defined in~\eqref{def b},
we set $x=1$ and $\eta= \widehat{h}$ in the explicit formula~\eqref{equation explicit formula}, resulting in the identity
\begin{equation}\label{explicite pour b}\begin{split}
  b (\chi;h)= \mathcal L_\eta(\tfrac 12) \delta_{\chi=\chi_0} +\widehat{h}(0)\log A(\chi) &+\frac{1}{2\pi}\int_{-\infty}^\infty\Big(\frac{\gamma_\chi'(\tfrac12+it)}{\gamma_\chi(\tfrac 12+it)}+\frac{\gamma_\chi'(\tfrac12-it)}{\gamma_\chi(\tfrac12-it)}\Big)h\Big(\frac{t}{2\pi}\Big) \,{\rm d}t\\ &- \psi_{\widehat{h}}(1;L/K,\chi)-\psi_{\widehat{h}}(1;L/K,\overline \chi)+O_h(1).
\end{split}\end{equation}

We have already seen in the proof of Lemma~\ref{lemma explicit formula} that the contribution of the gamma factors is $\ll \chi(1)$. Moreover, we have the bound
\begin{equation}\begin{split}
\psi_{\widehat{h}}(1;L/K,\chi) &\ll_h \chi(1)\sum_{\substack{\mathfrak p\triangleleft \mathcal O_K \\ m\geq 1 }}  \frac{\log (\mathcal N\mathfrak p)}{\mathcal N \mathfrak p^{(1+\frac \delta 2) m}}\cr&\ll \chi(1) \sum_p \sum_{f\geq 1} \frac{\log (p^f)}{ p^{f(1+\frac \delta 2) }} \sum_{\substack{\mathfrak p\triangleleft \mathcal O_K\\ \p \mid p \\ f(\p/p)=f  }}  1  \ll_\delta \chi(1)[K:\Q].\end{split}
\label{equation borne psi chi}
\end{equation}
The first claimed bound follows. As for the second, it is a consequence of Odlyzko type bounds (see \emph{e.g.}~\cite{PM}*{Theorem 3.2}).
\end{proof}

The next step will be to obtain an average bound on $b_0(\chi;\widehat \eta^2)$. Precisely if $t\colon G\rightarrow \C$ is a class function and $\eta\in \mathcal S_\delta$, then we analyze in the following lemma the variance defined in~\eqref{defVLK}.

\begin{lemma}\label{lemme estimation nu}
Assume Artin's conjecture and the GRH for the Galois extension of number fields $L/K$, and let $\eta\in \mathcal S_\delta$. Then we have the estimate \begin{equation}\label{eq:formule var} \nu(L/K,t;\eta)= \alpha\big(|\widehat \eta |^2\big) \sum_{ \chi \in \Irr(G)} |\widehat t(\chi)|^2 \log A(\chi)+E(L/K,t;\eta)    +O_\eta \big( [K:\Q] \lambda_{1,2}(t) \big)\,,
\end{equation}
where\footnote{Note that only the first term of this minimum will be used in this paper - the second is present for future reference.}
\begin{multline}
 E(L/K,t;\eta) \ll_\eta \min\Big\{    [K:\Q] \, \lambda_{1,2}(t) \frac{\log ({\rm rd}_L+2)}{\log_2 ({\rm rd}_L+2)},   \Big(\max_{\chi\in \Irr(G)} \frac{|\widehat t(\chi)|^2}{\chi(1)}\Big)\frac{\log (d_L+2)}{\log_2 (d_L+2)}\Big\}.
\end{multline}
Moreover, we have the bounds
\begin{equation}
\label{equation third bound on nu}
    \begin{split}
 \alpha(|\widehat \eta |^2)  \lambda_{1,2}(t)\Big(1-S_t-&O_\eta \Big(\frac 1{\log_2 (\rd_L+2)}\Big)\Big) \leq \frac{\nu(L/K,t;\eta)}{[K:\Q] \log(\rd_L) } \\ &\qquad\leq \alpha(|\widehat \eta |^2)   \lambda_{1,2}(t)\Big(1+S_t+O_\eta \Big(\frac 1{\log_2 (\rd_L+2)}\Big)\Big).
 \end{split}
\end{equation}

\end{lemma}
\begin{proof}
First observe that by~\eqref{eq:estimindiv}, we have the estimate
$$\sum_{\chi \in \Irr(G)}|\widehat t(\chi)|^2 b(\chi,|\widehat \eta|^2)=  \alpha(|\widehat \eta|^2) \sum_{\chi \in \Irr(G)}|\widehat t(\chi)|^2 \log A(\chi) + O_\eta \big( [K:\Q] \,\lambda_{1,2}(t) \big)\,.$$
Then, we remove the contribution of real zeros as follows:
\begin{align*}
 \sum_{\chi \in \Irr(G)}|\widehat t(\chi)|^2( b(\chi,|\widehat \eta|^2)- b_0(\chi,|\widehat \eta|^2) ) & \ll_\eta \sum_{\chi \in \Irr(G)}|\widehat t(\chi)|^2 {\rm ord}_{s=\frac 12} L(s,L/K,\chi) \\
 & \ll_\eta  \sum_{\chi \in \Irr(G)}|\widehat t(\chi)|^2 \frac{\log (A(\chi)+2)}{\log_2 (A(\chi)+2)^{\frac 3{\chi(1)[K:\Q]}}},
 \end{align*}
by~\cite{IK}*{Proposition 5.21}.
The first bound on $E(L/K,t;\eta)$ then follows directly from Lemma~\ref{lemme preconducteurs}. As for the second, we have that
\begin{align*}
 \sum_{\chi \in \Irr(G)}|\widehat t(\chi)|^2 \big( b(\chi,|\widehat \eta|^2)- b_0(\chi,|\widehat \eta|^2) \big) &\ll_\eta \Big(\max_{\chi\in \Irr(G)} \frac{|\widehat t(\chi)|^2}{\chi(1)}\Big) \cdot {\rm ord}_{s=\frac 12} \zeta_L(s) \\& \ll_\eta  \Big(\max_{\chi\in \Irr(G)} \frac{|\widehat t(\chi)|^2}{\chi(1)}\Big) \frac{\log (d_L+2)}{\log_2 (d_L+2)},
 \end{align*}
thanks to the decomposition $\zeta_L(s)=\prod_{\chi \in \Irr(G)} L(s,L/K,\chi)^{\chi(1)}$ and~\cite{IK}*{Proposition 5.34}.
Finally,~\eqref{equation third bound on nu} follows from combining~\eqref{eq:formule var} with the bounds in Lemma~\ref{lemme:SG}.
\end{proof}

In view of~\eqref{equation third bound on nu}, one may wonder if we can still produce a lower bound if $S_{t}$ is close to~$1$. In the next two lemmas we show that in this case we can still estimate $b_0(\chi,h)$ in terms of $\log A(\chi)$. The idea here is that if $\widehat{\eta}$ does not vanish on an interval containing sufficiently many imaginary parts of $L$-function zeros then we can deduce the required estimate. For $\chi\in\Irr(\Gal(L/K))$ we will denote, assuming GRH for $L(s,L/K,\chi)$,
$$
N(T,\chi)=\big\{\rho \colon 0<\Re(\rho) <1, |\Im(\rho)|\leq T,\, L(\rho,L/K,\chi)=0\big\}\,,\qquad (T\geq 0)\,.
$$

\begin{lemma}
\label{lemma CCM}
Assume Artin's conjecture and the GRH for the Galois extension of number fields $L/K$.
Let $G=\Gal(L/K)$ and $\chi\in \Irr(G)$. For all $T>0$ and all $0< \eps \leq 1$ one has
\begin{multline*}
N(T+\eps,\chi)-N(T,\chi) =  \frac {\eps}{\pi} \log \Big(A(\chi)\Big(\frac{T+\eps}{2\pi \e}\Big)^{\chi(1) [K:\Q]}\Big) \\+O\Big( \frac{\log ((A(\chi)+2) (4T+1)^{\chi(1)[K:\Q]}) }{\log_2 ((A(\chi)+2)^{\frac 3{\chi(1)[K:\Q]}} (4T+1))}+[K:\Q] \chi(1)\Big).
\end{multline*}
In particular, if $\eps \geq \kappa (\log_2 (T+3))^{-1}$ and
\begin{equation}
\label{eq:conditionSG}
 (1-S_t)^{-1} \leq \kappa^{-1} \eps \log_2 ({\rm rd}_L+2) \Big( 1+\frac{[K:\Q]\log T}{\log ({\rm rd}_L+2)} \Big),
 \end{equation}
where $\kappa>0$ is a large enough absolute constant, then we have the bound
$$ \sum_{\chi \in \Irr (G)} |\widehat t(\chi)|^2\big(N(T+\eps,\chi)-N(T,\chi)\big) \geq   \frac {\eps {\color{blue}}}{8\pi} \sum_{\chi \in \Irr (G)} |\widehat t(\chi)|^2 \log \Big(A(\chi)\Big(\frac{T+\eps}{2\pi \e}\Big)^{\chi(1) [K:\Q]}\Big). $$
In case $\rd_L \ll 1,$ then the assumption $\eps \gg \kappa (\log_2 (T+3))^{-1}$ is sufficient (\emph{i.e.}~\eqref{eq:conditionSG} is not required).
\end{lemma}

Note that the condition $\eps \gg \kappa(\log_2(T+3))^{-1}$ implies that $\eps$ or $T$ is large enough, which ensures that $N(T+\eps,\chi)-N(T,\chi) \neq 0$.

\begin{proof}
With notations as in~\eqref{eq:arch}, we combine~\cite{CCM}*{(4.1)} and~\cite{CCM}*{Theorem 5} to obtain
\begin{equation}\label{eq:zeroshauteurT}\begin{split}
N(T+\eps,\chi) -N(T,\chi) &=  \frac 1 \pi\int_{T<|t|<T+\eps} {\rm Re}\Big(  \frac{L'}{L}(\tfrac 12+it,\chi_{\infty})\Big)\d t \\ &+O\Big( \frac{\log ((A(\chi)+2) (4T+1)^{\chi(1)[K:\Q]}) }{\log_2 ((A(\chi)+2)^{\frac 3{\chi(1)[K:\Q]}} (4T+1))}\Big) +O\big([K:\Q]\chi(1)\big)\,.
\end{split}
\end{equation}

To evaluate the main term we use the computations~\cite{IK}*{(5.35) and (5.36)} in the context
of~\cite{CCM}*{(4.1)}. Precisely the factors of $L(s,\chi_\infty)$ have the following contribution in the range $[T,T+\eps]$ of imaginary parts of critical zeros:
\begin{align*}
&\frac{\eps}{ \pi}\log\Big(\frac{A(\chi)}{\pi^{{{[K:\Q]\chi(1)}}}}\Big)+\frac{[K:\Q]\chi(1)}{\pi}\Big((T+\eps)\log \frac{T+\eps}{2}-T\log \frac{T}{2}-\eps\Big)+O\big([K:\Q]\chi(1)\big)
\\ &\quad=
\frac{\eps }{\pi}\log\Big(\frac{A(\chi)}{\pi^{{{[K:\Q]\chi(1)}}}}\Big)+
\frac{[K:\Q]\chi(1)}{\pi} \eps\log\Big( \frac{T+\eps}{2\e}\Big)
+O\big([K:\Q]\chi(1)\big),
\end{align*}
which leads to the first estimate. In order to prove the second part of the statement, note that
\begin{multline*}
 \sum_{\chi \in \Irr(G)} |\widehat t(\chi)|^2  \frac{\log ((A(\chi)+2) (4T+1)^{\chi(1)[K:\Q]}) }{\log_2 ((A(\chi)+2)^{\frac 3{\chi(1)[K:\Q]}} (4T+1))}  \ll  \sum_{\chi \in \Irr (G)}  |\widehat t(\chi)|^2\frac{\log (A(\chi)+2) }{\log_2 ((A(\chi)+2)^{\frac 3{\chi(1)[K:\Q]}} )}\\
+ [K:\Q]\sum_{\chi\in\Irr(G)} \chi(1)|\widehat t(\chi)|^2 \frac {\log(4T+1)}{ \log_2 ((A(\chi)+2)^{\frac 3{\chi(1)[K:\Q]}})}.
\end{multline*}

Moreover, Lemma~\ref{lemme:SG} implies the bound (recall~\eqref{def St})
$$ \sum_{\chi \in \Irr(G)} |\widehat t(\chi)|^2\log A(\chi) \geq (1-S_t) \log ({\rm rd}_L) \lambda_{1,2}(t) [K:\Q] \,.  $$

The stated lower bound then follows from~\eqref{eq:zeroshauteurT} and from Lemma~\ref{lemme preconducteurs} and Lemma~\ref{lemma bounds on Artin conductor}. Indeed the main term is greater than twice the error term under the stated assumption. Finally note that if
 $2\leq \rd_L \ll 1$, then Lemma~\ref{lemma bounds on Artin conductor} implies that  $ \log (A(\chi)+2) \asymp [K:\Q] \chi(1)$ which is sufficient to obtain the stated lower bound. The only case not covered by this condition, which corresponds to $L=K=\Q$, can be trivially handled separately.
\end{proof}

Building on Lemma~\ref{lemma CCM}, we can now deduce an estimate on $b_0(\chi,\widehat{\eta}^2)$
(recall~\eqref{def:b0eta}) in terms of $\log A(\chi)$ under a support condition on $\widehat{\eta}$.

\begin{lemma}
\label{lemme estimation nu avec Achi}
Assume Artin's conjecture and the GRH for the Galois extension of number fields $L/K$.
Let $G=\Gal(L/K)$
and let $\eps,T>0$ be such that $T\geq \kappa$ and $ \eps\geq \kappa(\log_2 (T+3))^{-1}$, where $\kappa>0$ is absolute and large enough. 
Assuming that $\widehat \eta$ does not vanish on $[T,T+\eps]$ \footnote{Recall that in~\eqref{eq:conditionSG} the constant $\kappa>0$ is absolute. Note moreover that if $\widehat \eta$ does not vanish, the condition on $\widehat{\eta}$ is always fulfilled with $\eps=\infty$.}, then we have
\begin{equation}\label{eq:asymp var}
\nu(L/K,t;\eta)  \asymp_\eta \sum_{ \chi \in \Irr(G)} |\widehat t(\chi)|^2 \log (A(\chi){+2} )  \,.
\end{equation}

\end{lemma}

\begin{proof}
By definition, we have the lower bound
$$ b_0(\chi;\widehat{\eta}^2) \geq \big(N(T+\eps,\chi)-N(T,\chi)\big) \min_{|t|\in [T,T+\eps]} |\widehat \eta|^2\, \gg_\eta  \log(A(\chi)+2),$$
by Lemma~\ref{lemma CCM} and our hypotheses on $\eps$ and $T$, which imply that the main term in this lemma dominates the error term. As a result,
$$\nu(L/K,t;\eta)  \gg_\eta   \sum_{ \chi \in \Irr(G)} |\widehat t(\chi)|^2 \log(A(\chi)+2).$$
The upper bound follows directly from~\eqref{eq:estimindiv}.
\end{proof}

\section{Proof of Theorems~\ref{theorem main} and~\ref{theorem main 2}: Induction and positivity} \label{section:artin}

In this section our main goal is to prove Theorem~\ref{theorem main} and Theorem~\ref{theorem main 2}. This will be carried out through an application of positivity in the explicit formula obtained in Lemma~\ref{lemma explicit formula} (positivity will circumvent the need for the LI hypothesis). Notice however that doing so directly with the Fourier decomposition (recall the definition~\eqref{eq:psi})
$$  \psi_{\eta}(x;L/K,t) = \sum_{\chi \in \Irr(G)} \widehat t(\chi) \psi_{\eta}(x;L/K,\chi) $$
will yield bounds which we believe not to be optimal (unless $K= \Q$). To obtain conjecturally optimal bounds, we will first apply the inductive property of Artin $L$-functions. This is the purpose of Lemma~\ref{proposition psi bridge}. The following step, Lemma~\ref{lemme formule explicite Mn}, will consist in approximating the moment we study $\widetilde M_n(U,L/K;t,\eta,\Phi)$ by the quantity $\widetilde D_n(U,L/K;t,\eta,\Phi)$ which involves zeros of Artin $L$-functions. A lower bound for $\widetilde D_n(U,L/K;t,\eta,\Phi)$ will be produced in Lemma~\ref{lemme combinatoire} by combining two preparatory results: a combinatorial inequality which we believe is of intrinsic interest (Lemma~\ref{lemma multiplicites}) and a statement which is more representation theoretic in nature and deals with $L$-function zeros relevant to the moment $\widetilde M_n(U,L/K;t,\eta,\Phi)$ (Lemma~\ref{lemma pre combinatoire}).

We recall that $L/K$ is a Galois extension of number fields of Galois group $G$, and $t\colon G\rightarrow \C$ is a class function. If $F$ is a subfield of $K$ such that $L/F$ is Galois of group $G^+$, then we form the class function on $G^+$ induced by $t$ in the following way
$$
t^+={\rm Ind}_G^{G^+}(t)\colon G^+ \rightarrow \C\,,\qquad
 t^+(g)= \sum_{\substack{aG\in G^+/G \colon \\a^{-1} g a \in G }} t(a^{-1} g a)\,\,\,\, (g\in G^+)\,.
$$

Through this section, one should keep in mind that if we assume Artin's conjecture for~$L/\Q$, then we expect in most cases to obtain the best bounds by selecting $F=\Q$. On the other extreme, one may always take $F=K$ and obtain non-trivial bounds.

\begin{lemma}
\label{proposition psi bridge}
Let $L/K/F$ be a tower of number fields for which $L/F$ is Galois, let $G=\Gal(L/K)$ and $G^+=\Gal(L/F)$. For $\eta\in \mathcal S_\delta$ and for any class function $t\colon G\rightarrow \C$, we have the identity
\begin{equation}
 \psi_\eta(x;L/K,t) = \psi_\eta(x;L/F, t^+).
 \label{equation psi bridge}
\end{equation}
As a consequence, for any $\Phi\in \mathcal U$ we have the identity
\begin{equation}\label{eq:transferM}
\widetilde M_n(U,L/K;t,\eta,\Phi) = \widetilde M_n(U,L/F; t^+,\eta,\Phi).
\end{equation}
\end{lemma}

\begin{proof}
The equality~\eqref{equation psi bridge} is stated and proved in~\cite{FJ}*{Proposition 3.11}. As for~\eqref{eq:transferM}, it is a consequence of~\eqref{equation psi bridge} combined with~\cite{FJ}*{Lemma 3.15} and the equality $\widehat{t^+}(1)= \widehat t(1)$, which is a straightforward application of Frobenius reciprocity.
\end{proof}


We now approximate the moment $\widetilde M_n(U,L/K;t,\eta,\Phi)$ by a sum over zeros of Artin $L$-functions. If $L/F$ is a Galois extension of group $G^+$, then we define for every integer $n\geq 1$
\begin{equation}
\label{equation definition Dtilde}
\begin{split}  
 \widetilde D_n(U,L/F;t,\eta,\Phi):= 
 \frac{(-1)^n}{2 \int_0^\infty \Phi } &\sum_{ \chi_1,\ldots,\chi_n \in \Irr(G^+)} \Big(\prod_{j=1}^n \widehat t(\chi_j)\Big) \\ &\times \sum_{\gamma_{\chi_1},\ldots,\gamma_{\chi_n}\neq 0}  \widehat \Phi\Big(\frac U {2\pi }\big(\gamma_{\chi_1}+\cdots+\gamma_{\chi_n}\big)\Big)\prod_{j=1}^n\widehat \eta\Big(\frac{\gamma_{\chi_j}}{2\pi }\Big),\end{split}
 \end{equation}
 where $\gamma_{\chi_1}, \dots ,\gamma_{\chi_n}$ run over the imaginary parts of the non-trivial zeros of the Artin $L$-functions $L(s,L/F,\chi_1),\dots ,L(s,L/F,\chi_n)$.

\begin{lemma}
\label{lemme formule explicite Mn}

Let $L/K/F$ be a tower of number fields in which $L/F$ is a Galois extension satisfying Artin's conjecture and GRH. Let $t\colon\Gal(L/K) \rightarrow \C$ be a class function and let $t^+:= {\rm Ind}_{\Gal(L/K)}^{\Gal(L/F)} t$. Then for $\eta \in \mathcal S_\delta$, $\Phi\in\mathcal U$, and $n\in\Z_{\geq 1}$ we have the estimate
$$
\widetilde M_n(U,L/K;t,\eta,\Phi)=
\widetilde  D_n(U,L/F;t^+,\eta,\Phi)+O\Big( \frac {(\kappa_\eta [F:\Q]\, \lambda_{1,1}(t^+) \log({\rm rd}_L+2))^n}U \Big),
$$
where $\kappa_\eta>0$ is a constant which depends only on $\eta$.
\end{lemma}

\begin{proof}
Let $G^+=\Gal(L/F)$, and recall that by Lemma~\ref{proposition psi bridge}, one has
$$
\widetilde M_n(U,L/K;t,\eta,\Phi)=\widetilde M_n(U,L/F; t^+,\eta,\Phi)\,.
$$ 
Combining the Fourier decomposition 
$$ \psi_\eta(\e^u;L/F,t^+) = \sum_{\chi \in \Irr(G^+)} {\widehat {t^+}(\chi)}\psi_\eta(\e^u;L/F,\chi)$$
and Lemma~\ref{lemma explicit formula} results in the estimate (recall that Frobenius reciprocity implies $\widehat{t^+}(1)= \widehat t(1)$)
\begin{equation}
 \label{equation formule explicite t}
 \begin{split}
 \psi_\eta(\e^u;L/F,t^+) =&  \widehat t(1) x^{\frac 12} \mathcal L_\eta(\tfrac 12) -\sum_{\chi \in \Irr(G^+)} { \widehat {t^+}(\chi)}\sum_{\gamma_\chi} \e^{i \gamma_\chi u} \widehat \eta\Big( \frac {\gamma_\chi}{2\pi } \Big)\\ &+O_\eta\Big( \e^{-\frac u2}  \sum_{\chi \in \Irr(G^+)} |\widehat {t^+}(\chi)|\log (A(\chi){+2)} \Big).\end{split}
 \end{equation}
By Lemma \ref{lemma bounds on Artin conductor},  the error term is $\ll_\eta \e^{-\frac u2} [F:\Q] \log({\rm rd}_L) \lambda_{1,1}(t^+)   .$ The claimed estimate follows from substituting this expression in the definition~\eqref{equation definition Mtilde} and applying the bound
$$ \sum_{\gamma_\chi} \e^{i \gamma_\chi u} \widehat \eta\Big( \frac {\gamma_\chi}{2\pi } \Big) \ll_\eta \log (A(\chi)+2), $$
which is a direct consequence of the Riemann-von Mangoldt formula (see \emph{e.g.}~\cite{IK}*{Theorem~5.8}).
\end{proof}

Our goal will be to apply positivity on the right hand side of~\eqref{equation definition Dtilde}. The idea here is that by our conditions on $\widehat \Phi,\widehat {t^+}$ and $\widehat \eta$, the quantity $ \widetilde D_n(U,L/F;t^+,\eta,\Phi)$ is a sum of positive terms. The rapid decay of $\widehat \Phi$ should imply that only the terms where $ \gamma_{\chi_1}+\dots+ \gamma_{\chi_n} $ is very small contribute substantially to the inner sum in~\eqref{equation definition Dtilde}. However, if the zeros enjoy on average the diophantine properties of ``random'' real numbers, then we expect this to be the case only when the $\rho_{\chi_j}$ come in conjugate pairs, that is for each $j$ there exists $\pi(j)$ such that~$\gamma_{\chi_j} = -\gamma_{\chi_{\pi(j)}}$. Moreover, we also believe that this should force $\chi_j=\overline{\chi_{\pi(j)}}$.  Those two facts follow from an effective version of the linear independence hypothesis for Artin $L$-functions (see~\cite{FJ}*{Introduction} for the precise statement). The positivity condition will allow us to circumvent this hypothesis.

Let us first establish the following combinatorial result.

\begin{lemma}
\label{lemma multiplicites}
Let $\Gamma\subset \R_{>0}$ be a countable multiset, and let $\bfa=\{a_{\gamma}\}_{\gamma\in\Gamma\cup -\Gamma}$ be a sequence of complex numbers such that $a_{-\gamma}=\overline{a_\gamma}$ and moreover $ \sum_{\gamma\in \Gamma} |a_\gamma|^2 <\infty$, where by convention sums over $\gamma\in\Gamma$ take multiplicities into account. Define
$$S_{2\ell}(\bfa):= \sum_{\substack{\gamma_{1},\dots, \gamma_{\ell}\in \Gamma,\gamma'_{1},\dots, \gamma'_{\ell}\in -\Gamma
\\ \forall \gamma \in \R, \# \{ j : \gamma_j =\gamma \}= \# \{ j : \gamma'_j =-\gamma \} }}  \prod_{j=1}^{\ell} a_{\gamma_{j}}a_{ \gamma'_{j}}.$$
Then, $S_{2\ell}(\bfa) \in \mathbb R$, and moreover for every positive integer $\ell$, we have the inequality
\begin{equation}
\label{equation definition S2l}\begin{split}
 S_{2\ell}(\bfa)\geq  \ell! \Big(\sum_{\gamma\in\Gamma} |a_{\gamma}|^2 \Big)^{\ell-1}\max\Bigg\{ \sum_{\gamma\in\Gamma} |a_{\gamma}|^2 - \ell!\ell(\ell-1)M^2\e^{1/\ell}, 0\Bigg\}\,,
\end{split}
\end{equation}
where $M:= \sup \big\{|a_\gamma|\colon \gamma\in\Gamma\big\}$.
\end{lemma}

\begin{remark}\label{rem:ell=1}
For $\ell=1$, note that $\ell!\ell(\ell-1)M^2\e^{1/\ell}=0$. In fact, in this case we have  $S_2({\bfa})=\sum_{\gamma\in\Gamma} m_{\gamma} |a_{\gamma}|^2\geq \sum_{\gamma\in\Gamma}  |a_{\gamma}|^2$, where $m_\gamma
$ is the multiplicity of $\gamma$ in $\Gamma$. Indeed, by definition $m_{-\gamma}=m_{\gamma}$.
\end{remark}

\begin{proof}[Proof of Lemma~\ref{lemma multiplicites}] By Remark~\ref{rem:ell=1}, we may assume that $\ell\geq 2$. 
For any integer $r\geq 1$ and any $r$-tuple $\bfn=(n_1,\ldots,n_r)\in \mathbb N^r$, which is a partition of $\ell$ in the sense that $n_i\leq n_{i+1}$ for all $i$, and $\sum_in_i=\ell$, we denote by $s_1$ the number of indices $i\geq 1$ such that $n_i=n_1$, and inductively by $s_j$ the number of indices $i$ such that $n_i=n_{s_{j-1}+1}$. Note that if $t$ is the ``number of distinct parts'' in the partition $(n_1,\ldots,n_r)$ of $\ell$, in particular $s_t=\#\{i\colon n_i=n_r\}$, then one has $s_1+\cdots+s_t=r$. We set 
$$c(\bfn)=c(n_1,\ldots, n_r)=\binom{\ell}{n_1,\ldots, n_r} \frac{1}{s_1!\cdots s_t!}\,.
$$
where we recall the definition of the multinomial coefficient
$$ \binom{\ell}{ n_1,\ldots, n_r} =\frac{\ell !}{n_1!\cdots n_r!}\,.$$
In particular $c(1,\ldots, 1)=1$ since in this case $t=1$ and $s_1=r=\ell$.

\medskip
For every $\gamma\in \Gamma$, let  $m_\gamma
$ be the multiplicity of $\gamma$ in $\Gamma$.
On one hand we have the following expansion (here we use the notation $\sums$ to denote a sum ``without multiplicity'') 
\begin{equation}\label{puissance l de la somme}
\Big(\sum_{\gamma\in \Gamma} |a_{\gamma}|^2 \Big)^{\ell}=\Big(\sums_{\gamma\in \Gamma} m_{\gamma}|a_{\gamma}|^2 \Big)^{\ell}=\sum_{\substack{n_1+\cdots+n_r=\ell\\ n_1\leq n_2\leq \ldots \leq n_r}}c(\bfn)
\sums_{\substack{\gamma_{1},\dots, \gamma_{r}\in \Gamma  \\ \forall i\neq j , \gamma_{i}\neq \gamma _{j}}} \prod_{j=1}^{r}m_{\gamma_j}^{n_j} |a_{\gamma_{j}}|^{2n_j}.
\end{equation}
Here, we have used the fact that for a given $(n_1,\ldots,n_r)\in \mathbb N^r$ such that $n_1+\dots+n_r=\ell$, the number of permutations of the $n_j$'s such that $n_1\leq n_2\leq \ldots \leq n_r$ is exactly $s_1!\cdots s_t!$.

On the other hand we have
\begin{equation}\label{eq:expandS2ell}\begin{split}
    S_{2\ell}(\bfa)&=\sum_{\substack{n_1+\cdots+n_r=\ell\\ n_1\leq n_2\leq \ldots \leq n_r}}c(\bfn) 
\sums_{\substack{\gamma_{1},\dots, \gamma_{r}\in \Gamma  \\ \forall i\neq j , \gamma_{i}\neq \gamma _{j}}} \prod_{j=1}^{r}m_{\gamma_j}^{n_j}  a_{\gamma_{j}}^{n_j}
\sums_{\substack{\gamma'_{1},\dots, \gamma'_{\ell}\in \Gamma  \\ \forall i , \#\{j\leq \ell \,:\, \gamma'_{j}=-\gamma_{i}\}=n_i}} \prod_{j=1}^{\ell}m_{\gamma_j}^{n_j} a_{-\gamma_{j}}^{n_j},
\cr&=\sum_{\substack{n_1+\cdots+n_r=\ell\\ n_1\leq n_2\leq \ldots \leq n_r}}c(\bfn)\binom{\ell}{ n_1,\ldots, n_r}
\sums_{\substack{\gamma_{1},\dots, \gamma_{r}\in \Gamma  \\ \forall i\neq j , \gamma_{i}\neq \gamma _{j}}} \prod_{j=1}^{r}m_{\gamma_j}^{2n_j} |a_{\gamma_{j}}|^{2n_j}
\end{split}
\end{equation}
which is a real number. The additional factor $\binom{\ell}{ n_1,\ldots, n_r}$ comes from the number of the sets $\#\{j\leq \ell \,:\, \gamma'_{j}=-\gamma_{i}\}=n_i.$ Since the multiplicities $m_\gamma$ are positive integers, the contribution of $\bfn=(1,\ldots,1)$ to the right hand side of~\eqref{eq:expandS2ell} admits the lower bound
\begin{equation}\label{eq:LB11}
\ell!
\sums_{\substack{\gamma_{1},\dots, \gamma_{\ell}\in \Gamma  \\ \forall i\neq j , \gamma_{i}\neq \gamma _{j}}} \prod_{j=1}^{\ell}m_{\gamma_j}^2 |a_{\gamma_{j}}|^{2}\geq
\ell!
\sums_{\substack{\gamma_{1},\dots, \gamma_{\ell}\in \Gamma  \\ \forall i\neq j , \gamma_{i}\neq \gamma _{j}}} \prod_{j=1}^{\ell}m_{\gamma_j}  |a_{\gamma_{j}}|^{2}.
\end{equation}
Using~\eqref{puissance l de la somme} we see that the lower bound in~\eqref{eq:LB11} equals
$$
\ell! \Big(\sum_{\gamma\in\Gamma} |a_{\gamma}|^2 \Big)^{\ell}-\ell!
\sum_{\substack{n_1+\cdots+n_r=\ell\\ n_1\leq n_2\leq \ldots \leq n_r\\ n_r>1}}c(\bfn)
\sums_{\substack{\gamma_{1},\dots, \gamma_{r}\in \Gamma  \\ \forall i\neq j , \gamma_{i}\neq \gamma _{j}}} \prod_{j=1}^{r}m_{\gamma_j}^{n_j} |a_{\gamma_{j}}|^{2n_j}\,,
$$
and therefore we deduce from~\eqref{eq:expandS2ell} and~\eqref{eq:LB11} that
\begin{align} \label{eq:LBS2ell}
S_{2\ell}(\bfa)  \geq& \ell! \Big(\sum_{\gamma\in\Gamma} |a_{\gamma}|^2 \Big)^{\ell}
 \\&   \notag+
\sum_{\substack{n_1+\cdots+n_r=\ell\\ n_1\leq n_2\leq \ldots \leq n_r\\ n_r\geq 2}}c(\bfn)
\sums_{\substack{\gamma_{1},\dots, \gamma_{r}\in \Gamma  \\ \forall i\neq j , \gamma_{i}\neq \gamma _{j}}} \prod_{j=1}^{r}m_{\gamma_j}^{n_j} |a_{\gamma_{j}}|^{2n_j}\Bigg( \binom{\ell}{n_1,\ldots, n_r} \prod_{j=1}^{r}m_{\gamma_j}^{n_j}-\ell!\Bigg)
\\   
\geq &\ell! \Big(\sum_{\gamma\in\Gamma} |a_{\gamma}|^2 \Big)^{\ell} -\ell !S'_{2\ell}(\bfa)\,,
\notag
\end{align}
where we denote
$$S'_{2\ell}(\bfa):=
\sum_{\substack{n_1+\cdots+n_r=\ell\\ n_1\leq n_2\leq \ldots \leq n_r\\ n_r\geq 2}}c(\bfn)
\sums_{\substack{\gamma_{1},\dots, \gamma_{r}\in \Gamma  \\ \forall i\neq j , \gamma_{i}\neq \gamma _{j}
\\
\prod_{j=1}^{r}m_{\gamma_j}^{n_j}\leq n_1!\cdots n_r!}} \prod_{j=1}^{r}m_{\gamma_j}^{n_j} |a_{\gamma_{j}}|^{2n_j}\,.$$
Here we emphasize the extra condition $\prod_{j=1}^{r}m_{\gamma_j}^{n_j}\leq n_1!\cdots n_r!$ appearing in the index set of the inner sum. This is explained by the fact that $r$-tuples $\bfn$ such that $\prod_{j=1}^{r}m_{\gamma_j}^{n_j}> n_1!\cdots n_r!$ contribute a positive term to the second summand in~\eqref{eq:LBS2ell}.

To obtain an upper bound for $S'_{2\ell}(\bfa)$, we write
$$S'_{2\ell}(\bfa)=  \sum_{2\leq n_r\leq \ell }
\sum_{\substack{ (n_1,\dots,n_{r-1}) : \\n_1+\cdots+n_{r-1}=\ell-n_r\\ n_1\leq n_2\leq \ldots \leq n_r}}c(\bfn)
\sums_{\substack{\gamma_{1},\dots, \gamma_{r}\in \Gamma  \\ \forall i\neq j , \gamma_{i}\neq \gamma _{j}
\\
\prod_{j=1}^{r}m_{\gamma_j}^{n_j}\leq n_1!\cdots n_r!}} \prod_{j=1}^{r}m_{\gamma_j}^{n_j} |a_{\gamma_{j}}|^{2n_j}\,.  $$
Note that~\eqref{puissance l de la somme} can then be used for the partition $(n_1,\ldots, n_{r-1})$ of $\ell-n_r$ since we have
$$
c(\bfn)=c(n_1,\ldots, n_r)=\binom{\ell}{n_1,\ldots, n_r} \frac{1}{s_1!\cdots s_t!}\leq
  c(n_1,\ldots, n_{r-1})\binom{\ell}{n_r}\, .$$
We deduce that
\begin{equation}
S'_{2\ell}(\bfa)\leq
\sum_{2\leq n_r\leq \ell} \binom{\ell}{n_r}\Big(\sum_{\gamma\in\Gamma} |a_{\gamma}|^2 \Big)^{\ell-n_r}
\Bigg(\sums_{\substack{
\gamma\in\Gamma\\ m_{\gamma }^{n_r}\leq \ell !}}m_\gamma^{n_r} |a_{\gamma}|^{2n_r} \Bigg)\,.
\label{equation upper bound S'}
\end{equation}


Next we use the condition
$m_{\gamma}^{n_r}\leq \ell !$ in the index set of the innermost sum of~\eqref{equation upper bound S'}
as well as the inequality
$$\binom{\ell}{n_r}\leq \ell(\ell-1) \binom{\ell-2}{n_r-2},$$ to compute
\begin{align*}
S'_{2\ell}(\bfa)
&\leq   \ell !
\sum_{2\leq n_r\leq \ell} \binom{\ell}{n_r}\Big(\sum_{\gamma\in\Gamma} |a_{\gamma}|^2 \Big)^{\ell-n_r}
\Big(\sum_{\gamma\in\Gamma} |a_{\gamma}|^{2n_r} \Big)
\cr&\leq   \ell !\ell(\ell-1)M^2
\sum_{2\leq n_r\leq \ell} \binom{\ell-2}{n_r-2}\Big(\sum_{\gamma\in\Gamma} |a_{\gamma}|^2 \Big)^{\ell-n_r}M^{2(n_r-2)}
\Big(\sum_{\gamma\in\Gamma} |a_{\gamma}|^{2} \Big) \cr
&\leq   \ell !\ell(\ell-1)M^2
 \Big(M^2+\sum_{\gamma\in\Gamma} |a_{\gamma}|^2 \Big)^{\ell-1}
\,,
\end{align*}
where we have used the upper bound $|a_\gamma|^{2n_r}\leq M^{2n_r-2}|a_\gamma|^{2}$ and the binomial formula for the last step.
Plugging this into~\eqref{eq:LBS2ell} we deduce that
\begin{equation}
\label{equation-S2l-1}
\frac{S_{2\ell}(\bfa)}{\ell !}  \geq  \Big(\sum_{\gamma\in\Gamma} |a_{\gamma}|^2 \Big)^{\ell} - \ell!\ell(\ell-1)M^2
\Big(M^2+\sum_{\gamma\in \Gamma} |a_{\gamma}|^2 \Big)^{\ell-1}.
\end{equation}

To conclude, note that if $\sum_{\gamma\in \Gamma} |a_{\gamma}|^2\leq \ell(\ell-1)\ell!M^2$ then
we have obtained a trivial lower bound since
$S_{2\ell}(\bfa)\geq 0$ by~\eqref{eq:expandS2ell}.
However if $\sum_{\gamma\in \Gamma} |a_{\gamma}|^2> \ell(\ell-1)\ell!M^2$
then we have
$$\Big(M^2+\sum_{\gamma\in \Gamma} |a_{\gamma}|^2 \Big)^{\ell-1}\leq
\Big(\sum_{\gamma\in \Gamma} |a_{\gamma}|^2 \Big)^{\ell-1}\Big(1+\frac{1}{\ell(\ell-1)}\Big)^{\ell-1}
\leq
\Big(\sum_{\gamma\in \Gamma} |a_{\gamma}|^2 \Big)^{\ell-1}\e^{1/\ell}
$$
and therefore~\eqref{equation-S2l-1} yields in both cases the asserted lower bound
\begin{equation}
\label{lower bound S2l}
\frac{S_{2\ell}(\bfa)}{\ell !}  \geq  \Big(\sum_{\gamma\in\Gamma} |a_{\gamma}|^2 \Big)^{\ell-1}\max\Bigg\{ \sum_{\gamma\in\Gamma} |a_{\gamma}|^2 - \ell!\ell(\ell-1)M^2\e^{1/\ell}, 0\Bigg\}.
\end{equation}
\end{proof}

The following lemma is an application of Lemma~\ref{lemma multiplicites}. It makes use of the classification of irreducible characters $\chi$ of $G$ according to their \emph{Frobenius--Schur indicator} $\epsilon_2(\chi)$ (see \emph{e.g.}~\cite{Hup}*{Theorem 8.7}).

\begin{lemma}
\label{lemma pre combinatoire}
Let $L/F$ be a Galois extension of number fields for which Artin's conjecture and GRH hold. Define $G^+:=\Gal(L/F) $, and let $t^+\colon G^+ \rightarrow \R$ be a class function. For $\ell\in \mathbb N$, let $\eta\in \mathcal S_\delta$, $\psi\in \Irr(G^+)$, and let $\chi_1,\dots,\chi_{2\ell} \in \{ \psi, \overline{\psi} \}$. 
 If $\psi$ is unitary (that is, $\epsilon_2(\psi)=0$) then we have the estimate
\begin{equation}
 \sum_{\substack{\gamma_{\chi_1},\ldots,\gamma_{\chi_{\ell }}> 0 \\
 \gamma_{\chi_{\ell+1}},\ldots,\gamma_{\chi_{2\ell }}< 0 \\  \forall \gamma\in \R,\\ \# \{ k\leq 2\ell  : \chi_k\in \{ \psi, \overline{\psi} \} , \gamma_{\chi_k}= \gamma \} = \\ \# \{ k\leq 2\ell : \chi_k\in \{ \psi, \overline{\psi} \}, \gamma_{\chi_k}= -\gamma \}  }} \!\!\!\!\!\!\!\!\! \prod_{k=1}^{2\ell}\widehat \eta\Big(\frac{\gamma_{\chi_k}}{2\pi }\Big)
\geq \max\Big\{
\ell! b_0(\psi;|\widehat \eta|^2)^\ell -O_\eta\Big(\ell!^2\ell(\ell-1)   b_0(\psi;|\widehat \eta|^2)^{\ell-1} \Big),0\Big\},
\label{eq:psipsibar}
\end{equation}
where the $\gamma_{\chi_j}$ run through the multiset of imaginary parts of the zeros of $L(s,L/F,\psi)L(s,L/F,\overline\psi)$ (with multiplicity).

If $\psi$ is either orthogonal or symplectic (that is, $\epsilon_2(\psi)\in \{ \pm 1\}$), then
$$ \sum_{\substack{\gamma_{1},\dots,\gamma_\ell>0 \\ \gamma'_{1},\dots, \gamma'_{\ell} < 0 \\  \forall \gamma\in \R,\\ \# \{ k\leq  \ell  :  \gamma_{ k}= \gamma \} = \\ \# \{ k\leq  \ell :  \gamma'_{ k}= -\gamma \}  }} \prod_{k=1}^{\ell}\widehat \eta\Big(\frac{\gamma_{ k}}{2\pi }\Big)\widehat \eta\Big(\frac{\gamma'_{ k}}{2\pi }\Big)
\geq \max\Big\{2^{-\ell}\ell! b_0(\psi;|\widehat \eta|^2)^\ell -O_\eta\Big(2^{-\ell}\ell!^2\ell(\ell-1)    b_0(\psi;|\widehat \eta|^2)^{\ell-1} \Big),0\Big\},
$$
where the $\gamma_1,\dots,\gamma_\ell, \gamma'_1,\dots ,\gamma'_\ell$ run through the imaginary parts of the zeros of $L(s,L/F,\psi)$ (with multiplicity).
\end{lemma}

\begin{proof}
Let us start with the   case where $\psi$ is unitary. One of the difficulties comes from the fact that some $\gamma$ may satisfy
$L(\tfrac12+i\gamma,L/F,\psi)= L(\tfrac12+i\gamma,L/F, \overline{\psi })=0.$
We have
$$\# \big\{ k\leq 2\ell  : \chi_k\in\{\psi,\overline{\psi }\} , \gamma_{\chi_k}= \gamma \big\} = \\ \# \big\{ k\leq 2\ell : \chi_k\in\{\psi,\overline{\psi }\} , \gamma_{\chi_k}= -\gamma \big\}.$$
We define the multisets
$$\Gamma_1(\psi):=\big\{\gamma > 0\colon L(\tfrac12+i\gamma,L/F,\psi)= L(\tfrac12-i\gamma,L/F, {\psi })=0\big\} $$
and $$\Gamma_2(\psi):=\big\{\gamma> 0\colon L(\tfrac12+i\gamma,L/F,\psi)= 0,\quad L(\tfrac12-i\gamma,L/F, {\psi })\neq 0\big\}, $$  
so that $\Gamma_1(\psi)\cap \Gamma_2(\psi)=\varnothing$, $\Gamma_2(\psi)\cap \Gamma_2(\overline{\psi})=\varnothing$, and $\Gamma_1(\psi)\cup \Gamma_2(\psi)$ (respectively $\Gamma_1(\psi)\cup \Gamma_2(\overline{\psi})$) is a  multiset whose elements are the  positive imaginary parts of the non-trivial zeros of $L(s,L/F,\psi)$ (respectively $L(s,L/F,\overline{\psi})$). In the multiset $\Gamma_1(\psi)$, we define the multiplicity associated to $\gamma$ as the sum of the multiplicity of $\tfrac12+i\gamma$ for $L(s,L/F,\psi)$ and
the multiplicity of $\tfrac12-i\gamma$ for $L(s,L/F, \psi )$. Note that $\Gamma_1(\psi)=\Gamma_1(\overline{\psi})$.
Now, among the $\gamma_{\chi_k}$ in the sum on the left hand side of~\eqref{eq:psipsibar}, there are $2r$ elements in $\Gamma_1(\psi)$ where $0\leq r\leq \ell$. Thus we can write
\begin{align} &
\sum_{\substack{\gamma_{\chi_1},\ldots,\gamma_{\chi_{\ell }}> 0 \\
 \gamma_{\chi_{\ell+1}},\ldots,\gamma_{\chi_{2\ell }}< 0\\ \forall \gamma\in \R,\\ \# \{ k\leq 2\ell  : \chi_k\in\{\psi,\overline{\psi }\} , \gamma_{\chi_k}= \gamma \} = \\ \# \{ k\leq 2\ell : \chi_k\in\{\psi,\overline{\psi }\}, \gamma_{\chi_k}= -\gamma \}  }}  \prod_{k=1}^{2\ell}\widehat \eta\Big(\frac{\gamma_{\chi_k}}{2\pi }\Big)
\cr&=\sum_{r=0}^\ell \binom{\ell}{r}^2\Bigg(
\sum_{\substack{\gamma_{\chi_1},\ldots,\gamma_{\chi_{r }}\in {\Gamma_1(\psi)} \\
\gamma_{\chi_{r+1}},\ldots,\gamma_{\chi_{2r }}\in- {\Gamma_1(\psi)} \\  \forall \gamma\in \R,\\ \# \{ k\leq  2r  : \gamma_{\chi_k}= \gamma \} = \\ \# \{ k\leq 2r : \gamma_{\chi_k}= -\gamma \}  }}
\prod_{k=1}^{2r}\widehat \eta\Big(\frac{\gamma_{\chi_k}}{2\pi }\Big)\Bigg)\Bigg(
\sum_{\substack{\gamma_{\chi_1},\ldots,\gamma_{\chi_{ \ell- r }}\in \Gamma_2(\psi)\cup \Gamma_2(\overline{\psi}) \\ \gamma_{\chi_{ \ell- r +1}},\ldots,\gamma_{\chi_{2\ell-2r }}\in -(\Gamma_2(\psi)\cup \Gamma_2(\overline{\psi})) \\  \forall \gamma\in \R,\\ \# \{ k\leq 2\ell-2r  : \chi_k=\psi , \gamma_{\chi_k}= \gamma \} = \\ \# \{ k\leq 2\ell
-2r: \chi_k=\overline{\psi }, \gamma_{\chi_k}= -\gamma \}  }} \prod_{k=1}^{2\ell-2r}\widehat \eta\Big(\frac{\gamma_{\chi_k}}{2\pi }\Big)\Bigg).
\label{eq:Gamma1andGamma1bar}
\end{align}
Here, we use the convention that when $r=0$, the first sum is equal to $1$ whereas when $r=\ell$, the second sum is equal to $1$.

Reindexing the innermost sum in~\eqref{eq:Gamma1andGamma1bar}, we see that
$$\sum_{\substack{\gamma_{\chi_1},\ldots,\gamma_{\chi_{ \ell- r }}\in \Gamma_2(\psi)\cup \Gamma_2(\overline{\psi}) \\ \gamma_{\chi_{ \ell- r +1}},\ldots,\gamma_{\chi_{2\ell-2r }}\in -(\Gamma_2(\psi)\cup \Gamma_2(\overline{\psi})) \\  \forall \gamma\in \R,\\ \# \{ k\leq 2\ell -2r : \chi_k=\psi , \gamma_{\chi_k}= \gamma \} = \\ \# \{ k\leq 2\ell-2r : \chi_k=\overline{\psi }, \gamma_{\chi_k}= -\gamma \} }} \prod_{k=1}^{2\ell-2r}\widehat \eta\Big(\frac{\gamma_{\chi_k}}{2\pi }\Big)=S_{2\ell-2r}(\bfa), $$
where $S_{2\ell-2r}(\bfa)$ is defined in Lemma~\ref{lemma multiplicites}, with the choices
$$ \Gamma:=\Gamma_2(\psi)\cup \Gamma_2(\overline{\psi}); \qquad a_{\gamma}:= \widehat \eta \Big( \frac{\gamma}{2\pi } \Big).  $$
By  Lemma~\ref{lemma multiplicites}, it is
$$\geq
\max\Big\{
(\ell-r)! b_2(\psi;|\widehat \eta|^2)^{\ell-r} -O_\eta\Big((\ell-r)!^2(\ell-r)(\ell-r-1)   b_2(\psi;|\widehat \eta|^2)^{\ell-r-1} \Big),0\Big\},$$
where $ b_2(\psi;|\widehat \eta|^2)$ is the contribution of $\gamma\in \Gamma_2(\psi)\cup \Gamma_2(\overline{\psi})$ in $ b_0(\psi;|\widehat \eta|^2) $ so that
$$b_0(\psi;|\widehat \eta|^2)+b_0(\overline{\psi};|\widehat \eta|^2) =b_1(\psi;|\widehat \eta|^2) +b_2(\psi;|\widehat \eta|^2) $$ with
$$b_1(\psi;|\widehat \eta|^2)=2
 \sum_{\substack{\gamma\in     \Gamma_1(\psi) }} \Big|\widehat \eta\Big(\frac{\gamma}{2\pi  }\Big)\Big|^2; \qquad b_2(\psi;|\widehat \eta|^2)=2
 \sum_{\substack{\gamma\in  \Gamma_2(\psi)\cup \Gamma_2(\overline{\psi}) }} \Big|\widehat \eta\Big(\frac{\gamma}{2\pi  }\Big)\Big|^2.$$
In the same fashion, we may estimate the first bracketed sum on the right hand side of~\eqref{eq:Gamma1andGamma1bar} 
using Lemma~\ref{lemma multiplicites}, with the choices
$$ \Gamma:=  {\Gamma_1(\psi)} ; \qquad a_{\gamma}:= \widehat \eta \Big( \frac{\gamma}{2\pi } \Big).  $$ 
By the same argument, the first sum
is
$$\geq
\max\Big\{
r! b_1(\psi;|\widehat \eta|^2)^{r}  -O_\eta\big(r!^2r(r-1) b_1(\psi;|\widehat \eta|^2)^{r-1} \big),0\Big\},$$ 
Summing over $r$ yields the claimed estimate.

If $\psi$ is either orthogonal or symplectic, then the bound follows at once from Lemma~\ref{lemma multiplicites} with the choices
$$ \Gamma:=\big\{ \gamma > 0 : L(\tfrac 12 +i\gamma,L/F,\psi)=0 \big\}; \qquad a_{\gamma}:= \widehat \eta \Big( \frac{\gamma}{2\pi } \Big).  $$
\end{proof}

\begin{lemma}
\label{lemme combinatoire}
Let $L/F$ be a Galois extension of number fields for which Artin's conjecture and GRH hold. Define $G^+:=\Gal(L/F) $, and let $t^+\colon G^+ \rightarrow \R$ be a class function. Assume that $\widehat t^+ \in \mathbb R_{\geq 0} $ and let $\eta\in \mathcal S_\delta, \Phi\in \mathcal U$. For $m\in \mathbb N$, we have the lower bound
$$
  \widetilde  D_{2m}(U,L/F;t^+,\eta,\Phi) \geq \mu_{2m} \nu(L/F,t^+;\eta)^m \Big( 1 + O_\eta\big(m^2m! w_4(L/F,t^+;\eta) \big)\Big),
$$
where we recall~\eqref{defVLK} and
\begin{equation}
w_4(L/F,t^+;\eta):= \frac{\sum_{ \chi \in \Irr(G^+)}  |\widehat t^+(\chi)|^4b_0(\chi;\widehat \eta^2)}{\Big(\sum_{ \chi \in \Irr(G^+)}  |\widehat t^+(\chi)|^2b_0(\chi;\widehat \eta^2) \Big)^2}.
\end{equation} 
\end{lemma}


\begin{proof}

Firstly, in~\eqref{equation definition Dtilde}, we may replace $\Irr(G^+)$ by $C_t:={\rm supp}(\widehat {t^+})\subset \Irr(G^+).$ For simplicity, let us write (since $t$ is real valued)
$$C_t=\big\{ \psi_1,\psi_2,\ldots,\psi_{r_1},\psi_{r_1+1},\overline{\psi_{r_1+1}},\psi_{r_1+2}, \ldots,\psi_{r_1+r_2},\overline{\psi_{r_1+r_2}}\big\},$$
where $\psi_1,\dots \psi_{r_1}$ are real and $\psi_{r_1+1},\dots ,\psi_{r_1+r_2}$ are complex. Note that $C_t$ depends only on $G$ and $t$, and $r_1+2r_2=|C_t|$.
Given a vector $\bchi=(\chi_1,\dots \chi_{2m}) \in (C_t)^{2m}$, define
$$ E_j(\bchi):= \big\{1\leq  k \leq 2m \colon \chi_k \in\{ \psi_j,\overline{\psi_j} \}  \big\} \qquad (1\leq j \leq r_1+r_2), $$
$\ell_j(\bchi):=|E_j(\bchi)|$. Note that $ \sum_{j=1}^{r_1+r_2} \ell_j(\bchi) = 2m$.

Secondly, by positivity of $\widehat{t^+}$ and $\widehat \eta$, we may obtain a lower bound on $\widetilde D_{2m}(U,L/F;t^+,\eta,\Phi)$ by restricting the sum over characters to those $\bchi=(\chi_1,\dots ,\chi_{2m})$ that are elements of $(C_t)^{2m}$ and $(\gamma_{\chi_1},\dots,\gamma_{\chi_{2m}})$ for which
{for any  $j\leq r_1+r_2$ and $ \gamma \in \mathbb R$ we have
$$
\big|\big\{ k\in E_j(\bchi) \colon \chi_k\in\{\psi_j,\overline{\psi_j}\}, \gamma_{\chi_k}= \gamma \big\}\big| =\big|\big\{ k\in E_j(\bchi) \colon \chi_k\in\{\psi_j,\overline{\psi_j}\}, \gamma_{\chi_k}= -\gamma \big\}\big|.$$ }
 Finally, we may further impose that $k_j(\bchi):=\tfrac12\ell_j(\bchi)  \in \mathbb N$, and we may restrict the sum over characters to the subset $\mathcal C_{t,2m}$ of vectors of characters $\bchi = (\chi_1,\dots,\chi_{2m}) \in  C_{t}^{2m}$ which satisfy $|\{ \ell \leq 2m : \chi_\ell = \psi_j\}|=|\{ \ell  \leq 2m: \chi_\ell = \overline{\psi_j}\}|$, for every $r_1+1\leq j\leq r_1+r_2 $.
    We will also use the fact that for any $j\leq r_1+r_2$ and for all $(\chi_1,\dots ,\chi_{2k_j})$ and $(\gamma_{\chi_1},\dots ,\gamma_{\chi_{2k_j}})$ appearing in the index set of the double sum~\eqref{equation definition Dtilde}, 
    we have that
\begin{align*} \#\big\{ \ell\in E_j(\bchi) \big\}&= \sum_{\gamma\in\R_{>0}} \#\big\{ \ell \colon \chi_\ell  \in\{\psi_j,\overline{\psi_j}\} ,\gamma_{\chi_\ell}=\gamma\big\}+\sum_{\gamma\in\R_{<0}} \#\big\{ \ell \colon \chi_\ell{ \in\{\psi_j,\overline{\psi_j}\}},\gamma_{\chi_\ell}=\gamma\big\}  \\ &=2\sum_{\gamma\in\R_{>0}} \#\big\{ \ell \colon \chi_\ell{  \in\{\psi_j,\overline{\psi_j}\}},\gamma_{\chi_\ell}= \gamma\big\}.
\end{align*}

 As a result, one deduces the following lower bound:
\begin{multline*}
\widetilde  D_{2m}(U,L/F;t^+,\eta,\Phi)\geq \frac 1{2 \int_0^\infty \Phi } \times  \\  \sum_{ \substack{\bchi=(\chi_1,\ldots,\chi_{2m}) \in \mathcal C_{t,2m} \\ \forall j, k_j(\bchi) \in \mathbb N  }} \Big(\prod_{j=1}^{2m} \widehat t^+(\chi_j)\Big) \sum_{\substack{\gamma_{\chi_1},\ldots,\gamma_{\chi_{2m}}\neq 0 \\
\# \{ k\in E_j(\bchi) : \chi_k{  \in\{\psi_j,\overline{\psi_j}\}}, \gamma_{\chi_k}= \gamma \} = \\ \# \{ k\in E_j(\bchi) : \chi_k{  \in\{\psi_j,\overline{\psi_j}\}}, \gamma_{\chi_k}= -\gamma \} }}\widehat \Phi\Big(\frac U {2\pi }(\gamma_{\chi_1}+\cdots +\gamma_{\chi_{2m}})\Big)\prod_{j=1}^{2m}\widehat \eta\Big(\frac{\gamma_{\chi_j}}{2\pi }\Big).
\end{multline*}
At this point, we notice that the conditions in the inner sum automatically imply that $\gamma_{\chi_1}+\cdots+\gamma_{\chi_n}=0$, resulting in the bound
$$
\widetilde  D_{2m}(U,L/F;t^+,\eta,\Phi)\geq  \sum_{ \substack{\bchi=(\chi_1,\ldots,\chi_{2m}) \in \mathcal C_{t,2m} \\ \forall j,\, k_j(\bchi) \in \mathbb N  }} \Big(\prod_{j=1}^{2m} \widehat t^+(\chi_j)\Big) \sum_{\substack{\gamma_{\chi_1},\ldots,\gamma_{\chi_{2m}}\neq 0 \\ \forall j \leq r_1+r_2,  \forall \gamma\in\mathbb R, \\ \# \{ k\in E_j(\bchi) : \chi_k{  \in\{\psi_j,\overline{\psi_j}\}}, \gamma_{\chi_k}= \gamma \} = \\ \# \{ k\in E_j(\bchi) : \chi_k{  \in\{\psi_j,\overline{\psi_j}\}}, \gamma_{\chi_k}=- \gamma \}  }} 
\prod_{j=1}^{2m}\widehat \eta\Big(\frac{\gamma_{\chi_j}}{2\pi }\Big).
$$

Next we stratify the first sum according to the values assumed by 
$k_j(\bchi)$. Given a vector $\k=(k_1,\dots,k_{r_1+r_2}) \in \mathbb N^{r_1+r_2}$ such that $k_1+\dots+k_{r_1+r_2}=m$, we need to evaluate the sum
$$
D(\k  ):=  \sum_{ \substack{\bchi=(\chi_1,\ldots,\chi_{2m}) \in \mathcal C_{t,2m}\\ \forall j,\,  k_j(\bchi) =k_j  }} \Big(\prod_{j=1}^{2m} \widehat t^+(\chi_j)\Big)\sum_{\substack{\gamma_{\chi_1},\ldots,\gamma_{\chi_{2m}}\neq 0 \\ \forall j\leq r_1+r_2,\forall \gamma\in \R,\\ \# \{ k\in E_j(\bchi) : \chi_k{  \in\{\psi_j,\overline{\psi_j}\}}, \gamma_{\chi_k}= \gamma \} = \\ \# \{ k\in E_j(\bchi) : \chi_k{\in\{\psi_j,\overline{\psi_j}\}}, \gamma_{\chi_k}= -\gamma \}  }}  \prod_{j=1}^{2m}\widehat \eta\Big(\frac{\gamma_{\chi_j}}{2\pi }\Big).
$$
Now, note that since $t^+$ and $\widehat t^+$ are real-valued, we have that
$$ \widehat t^+(\chi) \widehat t^+(\overline \chi) =\widehat t^+(\chi)  \overline{ \widehat{ \overline t}(\chi) } = \big(\widehat t^+(\chi)\big)^2\,. $$
Hence, after reindexing we obtain the identity
$$ D(\k  )= \binom{2m}{2k_1,\dots,2k_{r_1+r_2}} \prod_{j=1}^{r_1+r_2}\Bigg((\widehat t^+(\psi_j))^{2k_j}\!\!\!\!\!\!\!\! \sum_{ \substack{(\chi_1,\ldots,\chi_{2k_j}) \in  \mathcal C_{t,2k_j}   \\ \forall \ell\leq { 2}k_j,\chi_\ell \in \{ \psi_j,\overline{\psi_j} \} }} \sum_{\substack{\gamma_{\chi_1},\ldots,\gamma_{\chi_{2k_j}}\neq 0 \\ \forall \gamma\in \R,\\ \# \{ k\leq 2k_j  : 
\gamma_{\chi_k}= \gamma \} = \\ \# \{ k\leq 2k_j : 
\gamma_{\chi_k}= -\gamma \}  }} \!\!         \prod_{k=1}^{2k_j}\widehat \eta\Big(\frac{\gamma_{\chi_k}}{2\pi }\Big)\Bigg).$$

Let us now evaluate the inner sum
$$\sigma_j(k_j):=\sum_{ \substack{(\chi_1,\ldots,\chi_{2k_j}) \in \mathcal C_{t,2k_j}   \\ \forall \ell\leq 2k_j,\,\chi_\ell \in \{ \psi_j,\overline{\psi_j} \} }} \sum_{\substack{\gamma_{\chi_1},\ldots,\gamma_{\chi_{2k_j}}\neq 0 \\ \forall \gamma\in \R,\\ \# \{ k\leq 2k_j  \colon {
} \gamma_{\chi_k}= \gamma \} = \\ \# \{ k\leq 2k_j \colon {
} \gamma_{\chi_k}= -\gamma \}  }}
\prod_{k=1}^{2k_j}\widehat \eta\Big(\frac{\gamma_{\chi_k}}{2\pi }\Big).
$$ 
Reindexing, we obtain the identity
$$ \sigma_j(k_j)=\binom{2k_j}{k_j} \sum_{\substack{\gamma_{1},\ldots,\gamma_{k_j}>0,\gamma'_{1} ,\dots ,\gamma'_{k_j}<0 \\  \forall \gamma\in \R,\\ \# \{ k\leq k_j : \gamma_{k}= \gamma \} = \\ \# \{ k_j< k\leq 2k_j : \gamma'_{k}= -\gamma \}  }}  \prod_{k=1}^{k_j}\widehat \eta\Big(\frac{\gamma_k}{2\pi }\Big)\eta\Big(\frac{\gamma'_{k}}{2\pi }\Big),
$$
where the $\gamma_j$ and the $ \gamma'_j$ are running over the positive (respectively negative) imaginary parts of the zeros of $L(s,L/K,\psi_j) L(s,L/K,\overline{\psi_j}) $.  Applying Lemma~\ref{lemma pre combinatoire}, we deduce that for~$j\geq r_1+1$ (\emph{i.e.} $\psi_j$ is unitary),
$$\sigma_j(k_j)
\geq  2^{k_j}\mu_{2k_j}  b_0(\psi_j;|\widehat \eta|^2)^{k_j}\max\Big\{ 1 -O_\eta\Big(\frac{{k_j}!{k_j}({k_j}-1)}{ b_0(\psi_j;|\widehat \eta|^2)} \Big),0\Big\},
$$
since
$$\binom{2k_j}{k_j}k_j!=2^{k_j}\mu_{2k_j}.$$

Now, if $\psi_j$ is either orthogonal or symplectic (\emph{i.e.} $j\leq r_1$), then we may fix the sign of the imaginary parts $\gamma_{\chi_j}$ and deduce that
$$ \sigma_j(k_j)=\binom{2k_j}{k_j} \sum_{\substack{\gamma_{1},\ldots,\gamma_{k_j}>0,\gamma'_{1} ,\dots ,\gamma'_{k_j}< 0 \\  \forall \gamma\in \R,\\ \# \{ k\leq k_j  : \gamma_{k}= \gamma \} = \\ \# \{ k_j<k\leq 2k_j : \gamma'_{k}= -\gamma \}  }}  \prod_{k=1}^{2k_j}\widehat \eta\Big(\frac{\gamma_{\chi_k}}{2\pi }\Big).
$$
We invoke Lemma~\ref{lemma pre combinatoire} once more and deduce the bound
$$\sigma_j(k_j)
\geq  \mu_{2k_j}b_0(\psi_j;|\widehat \eta|^2)^{k_j}\max\Big\{ 1 -O_\eta\Big(\frac{{k_j}!{k_j}({k_j}-1)}{ b_0(\psi_j;|\widehat \eta|^2)} \Big),0\Big\}.
$$

Putting everything together, we deduce the overall bound
\begin{multline}
 \label{equation three lines}
\widetilde  D_{2m}(U,L/F;t^+,\eta,\Phi) \geq \sum_{\substack{k_1,\ldots ,k_{r_1+r_2} \in \mathbb N\\ k_1+\ldots+k_{r_1+r_2} = m }} \binom{2m}{2k_1,\dots,2k_{r_1+r_2}} \\ \times \prod_{\ell =1}^{r_1} \Big(\mu_{2k_\ell} \widehat t^+(\psi_j)^{2k_\ell} b_0(\psi_\ell;|\widehat \eta|^2)^{k_\ell} \Big)
 \prod_{\ell=r_1+1}^{r_1+r_2} \Big( 2^{k_\ell}\mu_{2k_\ell} \widehat t^+(\psi_\ell)^{2k_\ell} b_0(\psi_\ell;|\widehat \eta|^2)^{k_\ell} \Big)
 \\ \times
 \prod_{\ell=1}^{r_1+r_2}\max\Big\{ 1 -O_\eta\Big(\frac{{k_\ell}!{k_\ell}({k_\ell}-1)}{ b_0(\psi_\ell;|\widehat \eta|^2)} \Big),0\Big\}.
\end{multline}

Let us first evaluate the main term in this expression. By the identity
$$ \binom{2m}{2k_1,\dots,2k_{r_1+r_2}}   \prod_{j=1}^{r_1+r_2} \mu_{2k_j} =\binom{m}{k_1,\dots,k_{r_1+r_2}} \mu_{2m}   $$
and the multinomial theorem, the main term  is equal to
$$ \mu_{2m} \Big(  \sum_{\ell =1}^{r_1} \widehat t^+(\psi_\ell)^2 b_0(\psi_\ell;|\widehat \eta|^2) + 2\sum_{\ell =r_1}^{r_1+r_2} \widehat t^+(\psi_\ell)^2 b_0(\psi_\ell;|\widehat \eta|^2) \Big)^{m}
= \mu_{2m} \nu(L/F,t^+;\eta)^m, $$
which is equal to the claimed main term.

 As for the error terms in~\eqref{equation three lines}, recall first that they vanish whenever $k_\ell\in \{ 0,1\}$ (see Remark~\ref{rem:ell=1}). Next we handle the contribution of indices $k_j\geq 2$ to the error terms. Using the identity
$$
 \prod_{\ell=1}^{r_1+r_2} \max\big\{ 1-x_\ell    ,0\big\}
 \geq 1-\sum_{j=1}^{r_1+r_2}  x_j \qquad (x_\ell \geq 0),   $$
 we see that we need to multiply the main term in~\eqref{equation three lines} by $$
\prod_{\ell=1}^{r_1+r_2} \max\Big\{ 1+ O\Big( \frac{k_\ell!k_\ell(k_\ell-1)}{b_0(\psi_\ell;|\widehat \eta|^2) } \Big),0  \Big\}  \geq  1+O\Big(\sum_{\substack{j =1\\ k_j \geq 2}}^{r_1+r_2}\frac{k_j !k_j (k_j -1)}{b_0(\psi_j ;|\widehat \eta|^2) }    \Big).
 $$
We obtain an error term which is
\begin{multline*}\ll\mu_{2m} \sum_{j=1}^{r_1+r_2} \frac 1{b_0(\psi_j,|\widehat \eta|^2)} \sum_{\substack{k_1,\ldots ,k_{r_1+r_2} \in \mathbb N\\ k_1+\ldots+k_{r_1+r_2} = m\\ k_j\geq 2 }} k_j!k_j(k_j-1)\binom{m}{k_1,\dots,k_{r_1+r_2}}   \\  \times \prod_{\ell =1}^{r_1} \Big(  \widehat t^+(\psi_\ell)^{2k_\ell} b_0(\psi_\ell;|\widehat \eta|^2)^{k_\ell}   \Big)
 \prod_{\ell=r_1+1}^{r_1+r_2} \Big( 2^{k_\ell} \widehat t^+(\psi_\ell)^{2k_\ell} b_0(\psi_\ell;|\widehat \eta|^2)^{k_\ell}\Big)  .
\end{multline*}
Finally, notice that
$$k_j(k_j-1)\binom{m}{k_1,\dots,k_{r_1+r_2}} =
m(m-1)\binom{m-2}{k_1,\dots, k_j-2,\ldots ,k_{r_1+r_2}} ,$$
and hence the error term above is
\begin{align*}
&\ll {m^2}m!\mu_{2m}  \Big(\sum_{j =1}^{r_1+r_2} \widehat t^+(\psi_j)^4b_0(\psi_j;|\widehat \eta|^2)  \Big) \Big( \sum_{\ell=1}^{r_1} \widehat t^+(\psi_\ell)^2 b_0(\psi_\ell;|\widehat \eta|^2) + 2\sum_{\ell =r_1}^{r_1+r_2} \widehat t^+(\psi_\ell)^2 b_0(\psi_\ell;|\widehat \eta|^2) \Big)^{m-2} \\
&\ll\mu_{2m} \nu(L/F,t^+;\eta)^{m-2} {m^2}m! \Big(\sum_{j =1}^{r_1+r_2} \widehat t^+(\psi_j)^4b_0(\psi_j;|\widehat \eta|^2) \Big)
. \end{align*}
\end{proof}

\begin{proof}[Proof of Theorem~\ref{theorem main}]
The claimed bound~\eqref{equation theorem main} follows from combining Lemmas~\ref{lemme formule explicite Mn} and~\ref{lemme combinatoire}.
\end{proof}

\begin{proof}[Proof of Theorem~\ref{theorem main 2}]

The first part follows from Lemmas~\ref{lemme:explicite2} and~\ref{lemme estimation nu avec Achi}. More precisely, the bound
$$ \sum_{ \chi \in \Irr(G^+)}  |\widehat t^+(\chi)|^4b_0(\chi;\widehat \eta^2)  \ll_\eta \sum_{ \chi \in \Irr(G^+)}  |\widehat t(\chi)|^4 \log(A(\chi)+2) $$
follows directly from Lemma~\ref{lemme:explicite2}.

Next~\eqref{equation lower bound on nu} follows from Lemma~\ref{lemme estimation nu}. We will also apply this lemma to prove the last claimed bound on $w_4(L/F,t^+;\eta)$. Note that by Lemmas~\ref{lemme:SG} and~\ref{lemme:explicite2} we have the upper bound
$$\sum_{ \chi \in \Irr(G^+)}  |\widehat t^+(\chi)|^4b_0(\chi;\widehat \eta^2) \ll_\eta \lambda_{1,4}(t^+)[F:\Q]\log (\rd_L+2).$$
Lemma~\ref{lemme estimation nu} then implies that
$$w_4(L/F,t^+;\eta) \ll_\eta \frac{1}{[F:\Q]\log (\rd_L+2)}\frac{\lambda_{1,4}(t^+)}{ \lambda_{1,2}(t^+)^2} \Big(1-S_{t^+}-O\Big( \frac 1{\log_2(\rd_L+2)} \Big)\Big)^{-2} .$$
Moreover,
we have the trivial bound
$$ \frac{\lambda_{1,4}(t^+)}{ \lambda_{1,2}(t^+)^2}\leq \frac{\lambda_{2,4}(t^+)}{ \lambda_{1,2}(t^+)^2} \leq 1.$$
The result follows.
\end{proof}

Finally, we prove Corollaries~\ref{corollary unweighted} and~\ref{corollary identity}.

\begin{proof}[Proof of Corollary~\ref{corollary unweighted}]

We will argue by contradiction. Assume otherwise that for all large enough $x$,
$$  \Big| \psi(x;L/K,t) - \widehat t(1) x \Big| \leq \eps(x) x^{\frac 12} C_{F,L,t^+}^{\frac 12},$$
where
$$C_{F,L,t^+}=  [F:\Q] \log(\rd_L) \lambda_{1,2}(t^+) \Big(1-S_{t^+}-  \frac A{\log_2(\rd_L+2)} \Big),$$
$A>0$ is an absolute and large enough constant
and $\eps(x)$ monotonically tends to zero as $x$ tends to $\infty.$ Let $\eta=\eta_0\star \eta_0$, where $\eta_0$ is a non-trivial smooth even function supported in~$[-1,1]$.
We then have that for large enough $x$,
\begin{align*}
\psi_\eta(x;L/K,t) -\widehat t(1) x^{\frac 12} \mathcal L_\eta(\tfrac 12) &= \int_0^\infty \frac{\eta(\log (\tfrac yx))}{ y^{\frac 12}}  \d \big( \psi(y;L/K,t) - \widehat t(1) y \big)\\
&= -\int_{\e^{-2} x}^{\e^2 x}  \frac{ \eta'(\log(\tfrac yx)) -\frac 12 \eta(\log(\tfrac yx))}{y^{\frac 32}} \big( \psi(y;L/K,t) - \widehat t(1) y \big) \d y \\
&\ll \eps(\e^{-2} x) C_{F,L,t^+}^{\frac 12}.
\end{align*}
Now, for any large enough $0<U_1<U_2$, this implies the bound
$$   \int_{U_1}^{U_2}  \big(\psi_\eta(\e^u;L/K,t)- 
\widehat t(1) \e^{\frac u2} \mathcal L_\eta(\tfrac 12) \big)^2 {\rm d} u \ll  \eps(\e^{-2} \e^{U_1})^2 (U_2-U_1) C_{F,L,t^+}.
$$
Moreover,~\eqref{equation psi bridge} implies that $$\psi_\eta(\e^u;L/K,t)=\psi_\eta(\e^u;L/F,t^+) = \sum_{\chi \in \Irr(G^+)}\widehat t^+(\chi) \psi_\eta(\e^u;L/F,\chi)\,.$$
We may apply Lemma~\ref{lemma explicit formula} in which we can bound the second term on the right hand side trivially (under GRH), resulting in the overall bound (recall that $\widehat t(1)=\widehat {t^+}(1)$)
$$ \psi_\eta(\e^u;L/K,t)  - \widehat t(1) \e^{\frac u2} \mathcal L_\eta(\tfrac 12)  \ll   \sum_{\chi \in \Irr(G^+)}|\widehat t^+(\chi)|  \log(A(\chi)+2) \ll \lambda_{1,1}(t^+) [F:\Q] \log(\rd_L+2).$$
This then implies that
 $$  \int_{0}^{U_1}  \big(\psi_\eta(\e^u;L/K,t)- \widehat t(1) \e^{\frac u2} \mathcal L_\eta(\tfrac 12) \big)^2 {\rm d} u \ll U_1 (\lambda_{1,1}(t^+)[F:\Q] \log(\rd_L+2))^2 .  $$
As a result, picking any even integrable function $\Phi$ supported in $[-1,1]$, we deduce that
$$M_2(U_2,L/K, t,\eta ,\Phi ) \ll \frac{U_1}{U_2} \big(\lambda_{1,1}(t^+)[F:\Q] \log(\rd_L)\big)^2 + \eps(\e^{-2} \e^{U_1})^2 \frac{U_2-U_1}{U_2}  C_{F,L,t^+}. $$
 Picking for instance $U_2 = U_1^{2}$, this will eventually contradict the lower bound in Corollary~\ref{corollary other moment} (combined with Theorem~\ref{theorem main 2}). Indeed, the bound $ S_{t^+} \leq 1- \kappa (\log_2 (\rd_L+2))^{-1}$ implies that~$\rd_L$ is large enough (since $\kappa$ itself is large enough), which in turns implies that $w_4(L/F,t^+,\eta)$ is small enough by Theorem~\ref{theorem main 2}.

 We now show that there exists a value $ \e^{U_1} \leq  x \leq  \e^{U_2}$ such that
 $$  \Big| \psi(x;L/K,t) -  \widehat t(1)  x \Big| \gg x^{\frac 12} C_{F,L,t^+}^{\frac 12},$$
 where $U_1= U$ and $U_2=\beta_{L,F,K,t}U$. 
 Assume otherwise that for all $\eps>0$ and for all extensions~$L/K$ and class functions $t$, there exists arbitrarily large values of $U$ (depending on $\eps$, $L/K$ and $t$) for which for all $x\in[\e^{U_1},\e^{U_2}]$,
 $$  \Big| \psi(x;L/K,t) -  {\widehat t(1)} x \Big| \leq \eps  x^{\frac 12} C_{F,L,t^+}^{\frac 12}.$$
One can deduce following the lines above that
 $$M_2(\e^{-2}U_2,L/K, t^+,\eta ,\Phi ) \ll \frac{U_1}{U_2} (\lambda_{1,1}(t^+)[F:\Q] \log(\rd_L))^2 +\eps   C_{F,L,t^+}. $$
 Once more, this will contradict Corollary~\ref{corollary other moment} if \begin{align*}
 U_2&>\kappa_2 [F:\Q] \log(\rd_L+2) \log_2(\rd_L+2) \lambda_{1,1}(t^+)^2/\lambda_{1,2}(t^+),\cr U_1&=\kappa_1U_2 \lambda_{1,2}(t^+)/\big([F:\Q] \lambda_{1,1}(t^+)^2\log(\rd_L+2)\log_2(\rd_L+2)\big) ,\end{align*}
 where $\kappa_2>0$ is large enough and $\kappa_1>0$ is small enough (both in absolute terms).
 \end{proof}

\begin{proof}[Proof of Corollary~\ref{corollary identity}]
The proof goes along the lines of that of Corollary~\ref{corollary unweighted}.
By Lemma~\ref{proposition psi bridge} applied to the tower $L/L/K$ and Lemma~\ref{lemme formule explicite Mn} applied to the trivial tower $L/L/L$,
\begin{align*}
\widetilde M_n(U,L/K;|G|{\bf 1}_{e},\eta,\Phi)& = \widetilde M_n(U,L/L; {\bf 1}_{e},\eta,\Phi)
\\
&=
\widetilde  D_n(U,L/L;{\bf 1}_{e},\eta,\Phi)+O\Big( \frac {(\kappa_\eta [K:\Q]\,  \log({\rm rd}_L+2))^n}U \Big).
\end{align*}
Moreover, by Lemma~\ref{lemme combinatoire},
$$
  \widetilde  D_{2m}(U,L/L;{\bf 1}_{e},\eta,\Phi) \geq \mu_{2m} \nu(L/L,{\bf 1}_{e};\eta)^m \Big( 1 + O_\eta\big(m^2m! w_4(L/L,{\bf 1}_{e};\eta)    \big)\Big),
$$
where
$$\nu(L/L,{\bf 1}_{e};\eta)= b_0(\chi_0;\widehat \eta^2) ; \qquad w_4(L/L,{\bf 1}_{e};\eta) = \frac 1{b_0(\chi_0;\widehat \eta^2)}. $$
Now, Lemma~\ref{lemme:explicite2} implies that
\begin{equation}
b (\chi;\widehat \eta^2)=\widehat{h}(0)\log {d_L} +O_\eta\big([L:\mathbb{Q}]\big)=\widehat{h}(0)\log {d_L} \Big(1+O\Big( \frac 1{\log(\rd_L+2)}\Big)\Big),
\end{equation}
resulting in the overall bound
\begin{multline*}
\widetilde M_{2m}(U,L/K;|G|{\bf 1}_{e},\eta,\Phi) \geq  \mu_{2m} (\widehat h(0) \log d_L)^m \Big(1+O_\eta\Big( \frac m{ \log(\rd_L+2)}+\frac{m^2m!}{\log (d_L+2)}\Big) \Big) \\ +O\Big( \frac {(\kappa_\eta \,  \log (d_L+2))^{2m}}U \Big).
\end{multline*}
The rest of the proof is similar.
\end{proof}

\section{Application to specific extensions and class functions: proofs}\label{section:prooftables}

This section is dedicated to the proofs of our results for specific Galois extensions, which were stated in~\S\ref{section examples}.

\subsection{$D_n$-examples} In this section, we prove Propositions~\ref{prop:tableDn} and~\ref{prop:dihedral}. With notation as in these statements, we recall that among the $\tfrac12(n+3) $ isomorphism classes of irreducible representations of $D_n$, exactly two  have degree 1: the trivial representation and the lift of the nontrivial character of $D_n/\langle \sigma\rangle$ which is defined by
$$
\psi(\sigma^j)=1\,,\qquad \psi(\tau\sigma^k)=-1\,.
$$
The remaining $\tfrac12(n-1) $ irreducible representations of $D_n$ have degree $2$;
the associated characters are given by
$$
\chi_h(\sigma^j)=2\cos(2\pi hj/n)\,,\qquad \chi_h(\tau\sigma^k)=0\,,
\qquad (h\in\{1,\ldots, \tfrac12(n-1)\})\,.
$$

\begin{proof}[Proof of Proposition~\ref{prop:tableDn}]
First note the following useful fact: for any integer $j$ such that~$n\nmid j$ we have
\begin{equation}\label{eq:jsum}
\frac 12\sum_{h=1}^{(n-1)/2}\chi_h(\sigma^j)=
\sum_{h=1}^{(n-1)/2} \cos \Big(\frac{2\pi hj}{n}\Big)=\frac{\sin (\pi j/2-\pi j/2n)}{\sin ( \pi j/n)}\cos \Big({\pi j\over 2}+{\pi j\over 2n}\Big)=-\frac 12\,.
\end{equation}

(1) If one considers $t=|D_n|{\bf 1}_e$, the indicator function of the neutral element of $D_n$, then~$\widehat{t}(\chi)=\chi(1)$ for any $\chi\in\Irr(D_n)$ and thus one computes for any $a\in D_n$: 
$$
\sum_{\chi\in\Irr(D_n)}\chi(a)|\widehat{t}(\chi)|^2=1+\psi(a)+4
\sum_{h=1}^{(n-1)/2}\chi_h(a)\,.
$$
If $a=e$, this sum equals $\lambda_{1,2}(t)=2+4(n-1)=4n-2$. If $a$ is in the conjugacy class of~$\tau$, then this sum vanishes, and finally if $a=\sigma^j$, then the sum equals $-2$ by~\eqref{eq:jsum}. Therefore~$S_t=1/(2n-1)$.

(2) Consider the class function $t={\bf 1}_{\{\sigma,\sigma^{-1}\}}$ (for which $\widehat t(\chi)=\chi(\sigma)/n$ for any $\chi\in\Irr(D_n)$). One has for any $a\in D_n$,
$$
\sum_{\chi\in\Irr(D_n)}\chi(a)|\widehat{t}(\chi)|^2=\frac{1+\psi(a)}{n^2}+\frac{4}{n^2}
\sum_{h=1}^{(n-1)/2}\chi_h(a)\Big(\cos\Big(\frac{2\pi h}{n}\Big)\Big)^2\,.
$$
If $a$ is conjugate to $\tau$ then this quantity vanishes. Also, for any $j'\in \{1,\ldots,\tfrac12(n-1) \}$, one has
$$
\sum_{\chi\in\Irr(D_n)}\chi(\sigma^{j'})|\widehat{t}(\chi)|^2=\frac{2}{n^2}+\frac{8}{n^2}
\sum_{h=1}^{(n-1)/2} \cos\Big(\frac{2\pi hj'}{n}\Big)\Big(\cos\Big(\frac{2\pi h}{n}\Big)\Big)^2\,.
$$
By linearizing the product on the right hand side, we see that the maximal value of the left hand side is attained at $j'=2$. Using~\eqref{eq:jsum} we can compute 
\begin{align*}
\sum_{\chi\in\Irr(D_n)}\chi(\sigma^{2})|\widehat{t}(\chi)|^2&=\frac{2}{n^2}+\frac{4}{n^2}\sum_{h=1}
^{(n-1)/2} \Big(\cos \Big(\frac{4\pi h}{n}\Big)+\cos \Big(\frac{4\pi h}{n}\Big)^2\Big)
= \frac{  n-2 }{n^2}\,.
\end{align*}

Moreover one has
\begin{align*}\lambda_{1,2}(t)=
\sum_{\chi\in\Irr(D_n)}\chi(1)|\widehat{t}(\chi)|^2&=\frac{2}{n^2}+\frac{4}{n^2}\sum_{h=1}
^{(n-1)/2} \Big(1+\cos \Big(\frac{4\pi h}{n}\Big)\Big)
=\frac{2(n-1)}{n^2 }\,.
\end{align*}

We conclude that $S_t=\frac{1-2/n}{2(1-1/n)}$.

(3) Finally consider $t=2{\bf 1}_e+{\bf 1}_{\{\sigma,\sigma^{-1}\}}$. Unlike ${\bf 1}_{\{\sigma,\sigma^{-1}\}}$, this class function has non-negative Fourier coefficients. Indeed one has
$$
\widehat{t}(1)=\widehat{t}(\psi)=\frac{2}{n}\,,
\qquad
\widehat{t}(\chi_h)=\frac{2}{n}\Big(1+\cos\Big(\frac{2\pi h}{n}\Big)\Big)=\frac 4n\cos\Big(\frac{4\pi h}{n}\Big)^2\,,\,\, \big(1\leq h\leq \tfrac12(n-1) \big)\,.
$$
Therefore one has for any $a\in D_n$,
$$
\sum_{\chi\in\Irr(D_n)}\chi(a)|\widehat{t}(\chi)|^2=\frac{4(1+\psi(a))}{n^2}+\frac{16}{n^2}
\sum_{h=1}^{(n-1)/2}\chi_h(a)\cos\Big(\frac{4\pi h}{n}\Big)^4\,.
$$
Using~\eqref{eq:jsum}, one finds that this sum equals $\tfrac2n(3-\tfrac4n)$ if $a=e$. If $a$ is in the conjugacy class of~$\tau$, the sum vanishes. If $a=\sigma^j$ and assuming $n\geq 5$, applying standard trigonometric identities as well as~\eqref{eq:jsum}, we see that this sum is equal to
\begin{multline*}
 \frac 2{n^2} +\frac{32}{n^2} \Bigg \{ \frac 14 \Big( -\frac 12  \cdot {\bf 1}_{j \not \equiv -4  \bmod n } + \frac{n-1}{2} \cdot {\bf 1}_{j \equiv -4 \bmod n } \Big) + \frac 14 \Big( -\frac 12  \cdot {\bf 1}_{j \not \equiv 4  \bmod n } + \frac{n-1}{2} \cdot {\bf 1}_{j \equiv 4 \bmod n } \Big) \\
 +\frac 1{16} \Big( -\frac 12  \cdot {\bf 1}_{j \not \equiv -8  \bmod n } + \frac{n-1}{2} \cdot {\bf 1}_{j \equiv -8 \bmod n } \Big) + \frac 1{16} \Big( -\frac 12  \cdot {\bf 1}_{j \not \equiv 8  \bmod n } + \frac{n-1}{2} \cdot {\bf 1}_{j \equiv 8 \bmod n } \Big)\Bigg\}.
\end{multline*}
Clearly, this quantity is maximized when $j=\pm 4 \bmod n$, in which case it is equal to $ \frac 2 n (2-\tfrac 4n)$.
Overall one concludes that $S_{t}\leq \frac{2-\frac 4n}{3-\frac 4n} <\frac 23$. 
\end{proof}

\begin{proof}[Proof of Proposition~\ref{prop:dihedral}]

 Set $t=|D_n|{\bf 1}_e$. One has
$$
\lambda_{1,1}(t)=\sum_{\chi\in \Irr(D_n)}\chi(1)^2=|D_n|=2n\,,\qquad
\lambda_{1,4}(t)=\sum_{\chi\in \Irr(D_n)}\chi(1)^5=2+\frac{32(n-1)}2=2(8n-7)\,.
$$
 We apply Theorem~\ref{theorem main 2} for $K=F=\Q$, and $L/\Q$ a $D_n$-extension. Therefore $G^+=G$ and $t^+=t$. Moreover $AC$ holds for $L$ since it is a supersolvable extension of $\Q$. Therefore applying Theorem~\ref{theorem main 2} we deduce
$$
w_4(L/\Q,t;\eta)\ll \frac 1{n\log{\rm rd}_L}\,.
$$
As for the variance, Theorem~\ref{theorem main 2} gives
$$
\left|\frac{\nu(L/\Q,t;\eta)}{\alpha(|\hat{\eta}|^2)(4n-2)\log{\rm rd}_L}-1\right|\leq \frac 1{2n-1}+O\Big(\frac{1}{ \log_2({\rm rd}_L+2)}\Big)\,.
$$
Putting this together, Theorem~\ref{theorem main} gives that for fixed $m\in \mathbb N$,
\begin{align*}
\widetilde M_{2m}(U,L/\Q;t,\eta,\Phi)&\geq \mu_{2m} \nu(L/K,t;\eta)^m\big(1+o_{\rd_L\to\infty}(1)\big) \\& \geq \mu_{2m}\Big(\alpha(|\hat{\eta}|^2)\Big(2-\frac 1n\Big)\log d_L\Big)^m \big(1+o_{\rd_L\to\infty}(1)\big)\,,
\end{align*}
as soon as $  (\log d_L)^m =o_{d_L\to\infty}(U)$.

\end{proof}

\subsection{Example of a radical extension}\label{subsec:rad}

In this section we prove Propositions~\ref{prop:tableradical} and~\ref{prop:radical}.

Notation is as in~\S\ref{subsec:radical}. The non trivial conjugacy classes of $G$ are:
$$
U:=\left\{\left(\begin{array}{cc}
1 & \star \\
0 & 1
\end{array} \right)\colon \star\neq 0\right\}\,,\qquad
T_c:=\left\{\left(\begin{array}{cc}
c & \star \\
0 & 1
\end{array} \right)\colon \star\in\F_p\right\}\,,\,\,(c\neq 1)\,.
$$
One has $|U|=p-1$ and $|T_c|=p$ for every $c\in\F_p\smallsetminus\{0,1\}$.
As for the characters of $G$, exactly $p-1$ of them have degree $1$: these are the lifts of Dirichlet characters $\chi$ modulo $p$
$$
\psi_\chi\colon \left\{\left(\begin{array}{cc}
c & d \\
0 & 1
\end{array} \right)\colon c\in \F_p^\times, d\in\F_p\right\}
\rightarrow (\Z/p\Z)^\times\xrightarrow{\chi}\C^\times\,,\qquad
\psi_\chi \left(\left(\begin{array}{cc}
c & d \\
0 & 1
\end{array} \right)\right)=\chi(c)\,.
$$
Finally $G$ has a unique irreducible character $\vartheta$ of degree $>1$. The character table  of $G$ summarizes the information:

\begin{center}
        \begin{tabular}{ |c | c | c | c | }
        \hline
             & $\{\rm Id\}$ & $U$ & $T_c,\, c\neq 1$\\
        \hline
            $\psi_\chi$ & $1$ & $1$ & $\chi(c)$ \\
        \hline
            $\vartheta$ & $p-1$ & $-1$ & $0$\\
        \hline
        \end{tabular}\,.
    \end{center}

    \begin{proof}[Proof of Proposition~\ref{prop:tableradical}]
Take $t=|G|{\bf 1}_{e}$, so that $\widehat{t}(\chi)=\chi(1)$ for all $\chi\in \Irr(G)$. Then for any $a\in G$, we have
$$
\sum_{\chi\in \Irr(G)}\chi(a)|\widehat{t}(\chi)^2|=\sum_{\chi \bmod p}\chi(a_{1,1})+(p-1)^2\vartheta(a)\,.
$$
(Here $a_{1,1}$ denote the coefficients in position $(1,1)$ of the matrix $a\in G$.) This sum vanishes at $a\in T_c$ for any $c$. The value of the sum at $a\in U$ is $-p(p-1)$ and finally, at $a=1$, the sum is $(p-1)+(p-1)^3$. Therefore
$$
S_t=\frac 1{p\big(1-\tfrac 2p+\tfrac 2{p^2}\big)}\,.
$$

Take $t=\vartheta$ which is real valued with $\hat{t}$ non negative. Then
$$
\sum_{\chi \in \Irr(G)} \chi(a)|\widehat \vartheta(\chi)|^2=\vartheta(a)\,,\qquad (a\in G)\,.
$$
Therefore $S_\vartheta=\frac{1}{p-1}$.
\end{proof}





\begin{proof}[Proof of Proposition~\ref{prop:radical}]
One has
$$
\lambda_{1,1}(|G|{\bf 1}_e)=\sum_{\chi\in\Irr(G)}\chi(1)^2=p(p-1)\,,\qquad 
\lambda_{1,1}(\vartheta)=\vartheta(1)=p-1\,.
$$
Moreover in the course of the proof of Proposition~\ref{prop:tableradical}, we have shown that
$$
\lambda_{1,2}(|G|{\bf 1}_e)=(p-1)(1+(p-1)^2)\,,\qquad \lambda_{1,2}(\vartheta)=p-1\,.
$$
Finally one computes
$$
\lambda_{1,4}(|G|{\bf 1}_e)=\sum_{\chi\in\Irr(G)}\chi(1)^5=(p-1)(1+(p-1)^4)\,,\qquad 
\lambda_{1,4}(\vartheta)=\vartheta(1)=p-1\,.
$$
Let $t$ be either $|G|{\bf 1}_e$ or $\vartheta$. We apply Theorem~\ref{theorem main 2} for $K=F=\Q$, and $L=K_{a,p}$. Therefore $G^+=G$ and $t^+=t$. Moreover $AC$ holds for $K_{a,p}$ since it is a supersolvable extension of $\Q$. Finally one has $d_L=|{\rm disc}(K_{a,p}/\Q)|=p^{p^2-2}a^{(p-1)^2}$ (see~\cite{Kom}*{end of the proof of the Theorem} and~\cite{Wes}*{\S 3.I}). Therefore
$$
\log d_L=p^2\log p(1+o_{p\to\infty}(1))\,,\qquad 
\log {\rm rd}_L=(1+o_{p\to\infty}(1))\log p\,. 
$$

For every $\eta\in\mathcal S_\delta$, the last bound of Theorem~\ref{theorem main 2} gives
$$
w_4(K_{a,p}/\Q,t;\eta)\ll \frac 1{p\log{\rm rd}_L}=\frac 1{p\log p}(1+o_{p\to\infty}(1))\,.
$$
As for the variance, Theorem~\ref{theorem main 2} gives
$$
\left|\frac{\nu(K_{a,p}/\Q,t;\eta)}{\alpha(|\hat{\eta}|^2)\lambda_{1,2}(t)\log{\rm rd}_L}-1\right|\leq S_t+O\Big(\frac{1}{ \log_2(p+2)}\Big)\,.
$$
Next we use the value of $S_t$ computed in Proposition~\ref{prop:tableradical}: $S_t=o_{p\to\infty}(1)$. Plugging these bounds into~\eqref{equation theorem main}, we conclude the proof.

\end{proof}

\subsection{Real parts of characters as class functions}

In this section we prove Proposition~\ref{prop:uniquechi}. We will need the following group theoretic preparatory result.

\begin{lemma}\label{lem:gp}
Let $G$ be a finite group and let $\rho\colon G\to {\rm GL}(V)$ be an irreducible finite dimensional complex representation of $G$. Let $\chi$ be the character of $\rho$ and let $a\in G$. We denote by $[a]$ the class of $a$ in $G/\ker\rho$. Then we have the following equivalences.
\begin{enumerate}
    \item $|\chi(a)|=\chi(1)$ if and only if $[a]\in Z(G/\ker\rho)$,
    \item $|\chi(a)+\overline{\chi(a)}|=2\chi(1)$ if and only if $[a]$ is an element of order $1$ or $2$ in $Z(G/\ker\rho)$.
\end{enumerate}
\end{lemma}

\begin{proof}
(1) First assume $|\chi(a)|=\chi(1)$. Since $\chi(a)$ is a sum of $\chi(1)$ roots of unity and by the triangle inequality, we obtain that $\rho(a)$ has a unique root of unity as eigenvalue. Being a semisimple matrix, we deduce that $\rho(a)$ is a scalar matrix, thus commutes with every element of ${\rm End}(V)$. Since $\rho$ induces a faithful representation of $G/\ker\rho$ with representation space $V$, we conclude that the class of $a$ in $G/\ker\rho$ lies in its center. Conversely, assume $[a]$ commutes with every element of $G/\ker\rho$. Then $\rho(a)$ commutes with every element of~${\rm End}(V)$. Since~$\rho$ is irreducible, Schur's lemma implies that $\rho(a)$ is a scalar matrix and thus $|\chi(a)|=\chi(1)$. 

(2) Since $|\chi(a)|\leq \chi(1)$, the equality $|\chi(a)+\overline{\chi(a)}|=2\chi(1)$ is equivalent to $\chi(a)=\pm \chi(1)$. By (1), this condition on $a$ implies that $[a]$ lies in the center of $G/\ker \rho$ with $\rho(a)$ a scalar matrix of trace $\pm\chi(1)$. In other words $\rho(a)=\pm{\rm Id}$, \emph{i.e.} $\rho(a^2)={\rm Id}$. Since $\rho$ induces a faithful representation of $G/\ker\rho$ this is in turn equivalent to $[a]$ having order at most $2$ in $Z(G^+/\ker\rho)$. The converse holds since if $[a]$ is an element of order at most $2$ in~$Z(G/\ker\rho)$, then $\rho(a)=\pm{\rm Id}$ and therefore $\chi(a)=\pm\chi(1)$.
\end{proof}

\begin{proof}[Proof of Proposition~\ref{prop:uniquechi}] 
Since $t^+=\frac{\chi+\overline{\chi}}{2}$, then $\widehat{t^+}(\psi)=\frac 12$ if $\psi\in\{\chi,\overline{\chi}\}$, and $\widehat{t^+}(\psi)=0$ for every other irreducible character of $G^+$. We deduce that $S_{t^+}=\max_{a\neq 1}|\chi(a)+\overline{\chi}(a)|/(2\chi(1))$. By Lemma~\ref{lem:gp}(2), we deduce that $S_{t^+}=1$ if and only if $Z(G/\ker\rho)$ has an element of order~$1$ or $2$. This is in turn equivalent to $\ker\rho=\{e\}$ and $|Z(G^+)|$ odd.

\end{proof}

We see that the particular case where $\Q=F=K$ and $G^+=\Gal(L/\Q)$ admits a faithful irreducible character $\chi$ and where $Z(G^+)$ has odd order is precisely that of~\S\ref{subsec:radical}.

\subsection{$S_n$-extensions}
In the section, we prove Proposition~\ref{proposition S_n}.

\begin{proof}[Proof of Proposition~\ref{proposition S_n}]

We begin by noting that following~\cite{FJ}*{Proof of Lemma 7.4}, one can show that Roichman's bound~\cite{Roi} combined with the hook-length formula imply that for any $\chi \in \Irr(S_n)$,
\begin{equation}
    \max_{{\rm id}\neq \pi \in S_n} \frac{\chi(\pi)}{\chi(1)} \leq \Big( \max\Big(q,\frac{ \log (k n!/\chi(1)) + \frac{2n}{\e}}{\log n!} \Big)\Big)^b, 
    \label{equation bound roichman}
\end{equation} 
where $0<q<1$, $k\geq 1$ and $b>0$ are absolute constants.  For simplicity, let us denote $t=t_{C_1,C_2}$. We will apply the bound~\eqref{equation bound roichman} on characters for which $\chi(1) \geq \|t\|_2(4p(n)^{\frac 12}\|t\|_1)^{-1}$. Note that
 $$ \|t\|^2_2 = \frac{n!}{|C_1|} +\frac{n!}{|C_2|};\qquad \|t\|_1 =2. $$
We may now apply Theorem~\ref{theorem main 2}, in the generalized form given in Remark~\ref{remark character removal}. Setting $$\Xi_{n;C_1,C_2}:=\big\{ \chi \in \Irr(S_n) : \chi(1)\geq \|t\|_2 (8p(n)^{\frac 12})^{-1} \big\} ,$$ it follows that for all large enough $n$,
\begin{align*}
    S_t(\Xi_{n;C_1,C_2}) & \leq \Big( \max\Big(q,\frac{ \log (kn!^{\frac 12} \min(|C_1|,|C_2|)^{\frac 12}) + \frac{2n}{\e}}{\log n!} \Big)\Big)^b \\ &\leq \max\Big(\theta_1, \Big( 1- \frac{\log (n!/\min(|C_1|,|C_2|))}{2\log n!} + \frac{2+o_{n \rightarrow \infty}(1)}{\e \log n }\Big)^b\Big) \\ 
    & \leq 1- \theta_2\frac{\log (n!/\min(|C_1|,|C_2|))}{2\log n!},
\end{align*} 
where $0<\theta_1<1$ and $\theta_2>0$ are absolute. We now claim that
$ \lambda_{1,2}(t,\Xi)  \gg \lambda_{1,2}(t).  $ To see this, we argue as in~\cite{FJ}*{Proposition 4.7}. We have the bound
$$ \lambda_{1,2}(t,\Irr(G)\smallsetminus\Xi) \leq  \frac{\|t\|_2 }{ 8 p(n)^{\frac 12}} \lambda_{0,2}(t) =  \frac{ \| t\|_2^3 }{ 8p(n)^{\frac 12}} ,  $$
by Parseval's identity in the form $\lambda_{0,2}(t) = \|t \|_2^2$. Moreover,~\cite{FJ}*{(111)} implies that 
$$\lambda_{1,2}(t)  \geq \frac{\| t\|_2^3}{2\sqrt 2 p(n)^{\frac 12} \| t\|_1 }, $$ and as a result we deduce that $$\lambda_{1,2}(t;\Xi_{n;C_1,C_2})  \gg \frac{\| t\|_2^3}{ p(n)^{\frac 12}}.$$ We can now apply Theorem~\ref{theorem main} to deduce the claimed bound. For the case $t=t_{C_1}$, the proof is identical.
\end{proof}



\section*{Acknowledgements}
The work of the third author was partly funded by the ANR through project FLAIR (ANR-17-CE40-0012). The authors would like to thank \emph{Villa La Stella} in Florence for their hospitality and excellent working conditions during a stay in March 2022 where substantial parts of this work were accomplished.


\begin{bibdiv}
\begin{biblist}

\bib{A}{article}{
   author={Artin, E.},
   title={Zur theorie der $L$-reihen mit allgemeinen gruppencharakteren},
   journal={Abh. Math. Sem. Univ. Hamburg},
   volume={8},
   date={1931},
   number={1},
   pages={292--306},
}


\bib{Be}{article}{
   author={Bella\"{\i}che, Jo\"{e}l},
   title={Th\'{e}or\`eme de Chebotarev et complexit\'{e} de Littlewood},
   journal={Ann. Sci. \'{E}c. Norm. Sup\'{e}r. (4)},
   volume={49},
   date={2016},
   number={3},
   pages={579--632},
}


\bib{BF1}{article}{
  author={de la Bret\`eche, R\'egis},
  author={Fiorilli, Daniel},
  title={On a conjecture of Montgomery and Soundararajan},
  journal={Math. Annalen},
   volume={381},
   date={2021},
   number={1},
   pages={575--591}
}




\bib{BF2}{article}{
  author={de la Bret\`eche, R\'egis},
  author={Fiorilli, Daniel},
  title={Moments of moments of primes in arithmetic progressions},
  date={2020},
  eprint={arXiv:2010.05944},
}

\bib{BFJ2}{article}{
  author={de la Bret\`eche, R\'egis},
  author={Fiorilli, Daniel},
  author={Jouve, Florent},
  title={Moments in the Chebotarev Density Theorem: non-Gaussian families},
   date={2022},
}


\bib{CCM}{article}{
   author={Carneiro, Emanuel},
   author={Chandee, Vorrapan},
   author={Milinovich, Micah B.},
   title={A note on the zeros of zeta and $L$-functions},
   journal={Math. Z.},
   volume={281},
   date={2015},
   number={1-2},
   pages={315--332}
}



\bib{FJ}{article}{
  author={Fiorilli, Daniel},
  author={Jouve, Florent},
  title={Distribution of Frobenius elements in families of Galois extensions},
  date={2020},
  eprint={hal-02464349},
}


\bib{FM}{article}{
  author={Fiorilli, Daniel},
  author={Martin, Greg},
  title={Inequities in the Shanks-R\'{e}nyi prime number race: an asymptotic
  formula for the densities},
  journal={J. Reine Angew. Math.},
  volume={676},
  date={2013},
  pages={121--212},
}

\bib{H7}{article}{
   author={Hooley, C.},
   title={On the Barban-Davenport-Halberstam theorem. VII},
   journal={J. London Math. Soc. (2)},
   volume={16},
   date={1977},
   number={1},
   pages={1--8},
}

\bib{Hup}{book}{
   author={Huppert, Bertram},
   title={Character theory of finite groups},
   series={De Gruyter Expositions in Mathematics},
   volume={25},
   publisher={Walter de Gruyter \& Co., Berlin},
   date={1998},
}



\bib{IK}{book}{
   author={Iwaniec, Henryk},
   author={Kowalski, Emmanuel},
   title={Analytic number theory},
   series={American Mathematical Society Colloquium Publications},
   volume={53},
   publisher={American Mathematical Society, Providence, RI},
   date={2004},
   pages={xii+615},
}

\bib{KF89}{book}{
   author={Kolmogorov, A. N.},
   author={Fomin, S. V.},
   title={Elements of the theory of functions and functional analysis},
   language={Russian},
   edition={6},
   note={With a supplement, ``Banach algebras'', by V. M. Tikhomirov},
   publisher={``Nauka'', Moscow},
   date={1989},
   pages={624},
}

\bib{Kom}{article}{
   author={Komatsu, Kenz\^{o}},
   title={An integral basis of the algebraic number field $\Q(\sqrt[t]{a},\sqrt
   [t]{ 1})$},
   journal={J. Reine Angew. Math.},
   volume={288},
   date={1976},
   pages={152--153},
}

\bib{LO}{article}{
   author={Lagarias, J. C.},
   author={Odlyzko, A. M.},
   title={Effective versions of the Chebotarev density theorem},
   conference={
      title={Algebraic number fields: $L$-functions and Galois properties},
      address={Proc. Sympos., Univ. Durham, Durham},
      date={1975},},
   book={publisher={Academic Press, London},},
   date={1977},
   pages={409--464},
}

\bib{Mar}{article}{
   author={Martinet, J.},
   title={Character theory and Artin $L$-functions},
   conference={
      title={Algebraic number fields: $L$-functions and Galois properties},
      address={Proc. Sympos., Univ. Durham, Durham},
      date={1975},
   },
   book={
      publisher={Academic Press, London},
   },
   date={1977},
   pages={1--87},
}


\bib{MV07}{book}{
   author={Montgomery, Hugh L.},
   author={Vaughan, Robert C.},
   title={Multiplicative number theory. I. Classical theory},
   series={Cambridge Studies in Advanced Mathematics},
   volume={97},
   publisher={Cambridge University Press, Cambridge},
   date={2007},
   pages={xviii+552},
}


\bib{PM}{article}{
   author={Pizarro-Madariaga, Amalia},
   title={Lower bounds for the Artin conductor},
   journal={Math. Comp.},
   volume={80},
   date={2011},
   number={273},
   pages={539--561},
}

\bib{Roi}{article}{
   author={Roichman, Yuval},
   title={Upper bound on the characters of the symmetric groups},
   journal={Invent. Math.},
   volume={125},
   date={1996},
   number={3},
   pages={451--485},
}

\bib{RS}{article}{
   author={Rubinstein, Michael},
   author={Sarnak, Peter},
   title={Chebyshev's bias},
   journal={Experiment. Math.},
   volume={3},
   date={1994},
   number={3},
   pages={173--197},
}




\bib{Se}{book}{
   author={Serre, Jean-Pierre},
   title={Linear representations of finite groups},
   series={Graduate Texts in Mathematics, Vol. 42},
   publisher={Springer-Verlag, New York-Heidelberg},
   date={1977},
   pages={x+170},
}




\bib{V}{article}{
   author={Viviani, Filippo},
   title={Ramification groups and Artin conductors of radical extensions of
   $\mathbb Q$},
   journal={J. Th\'{e}or. Nombres Bordeaux},
   volume={16},
   date={2004},
   number={3},
   pages={779--816},
}



\bib{Wes}{article}{
   author={Westlund, Jacob},
   title={On the fundamental number of the algebraic number-field $k(\root
   p\of m)$},
   journal={Trans. Amer. Math. Soc.},
   volume={11},
   date={1910},
   number={4},
   pages={388--392},
}

 \bib{Wi}{article}{
   author={Wintner, Aurel},
   title={On the distribution function of the remainder term of the prime
   number theorem},
   journal={Amer. J. Math.},
   volume={63},
   date={1941},
   pages={233--248},
}

\end{biblist}
\end{bibdiv}
\end{document}